\documentclass{amsart}

\usepackage{graphicx}

\usepackage{tikz}
\usetikzlibrary{arrows}
\usetikzlibrary{patterns}

\allowdisplaybreaks

\theoremstyle{plain}
\newtheorem{theorem}{Theorem}[section]
\newtheorem{proposition}[theorem]{Proposition}

\newtheorem{problem}[theorem]{Problem}

\theoremstyle{definition}
\newtheorem*{definition}{Definition}

\theoremstyle{remark}
\newtheorem{remark}[theorem]{Remark}

\begin{document}

\title[Surface bundles over surfaces]
{Surface bundles over surfaces with a fixed signature}

\author[Ju A Lee]{Ju A Lee}
\address{Ju A Lee \\ Department of Mathematial Sciences \\ Seoul National University \\ Seoul 151-747, Korea}
\email{jualee@snu.ac.kr}


\begin{abstract}
The signature of a surface bundle over a surface is known to be divisible by $4$. It is also known that the signature vanishes if the fiber genus $\leq 2$ or the base genus $\leq 1$. In this article, we construct new smooth $4$-manifolds with signature $4$ which are surface bundles over surfaces with small fiber and base genera. From these we derive improved upper bounds for the minimal genus of surfaces representing the second homology classes of a mapping class group.
\end{abstract}

\maketitle

\section{Introduction}

By a surface bundle over a surface we mean an oriented fiber bundle whose fibers and base are both compact, oriented $2$-manifolds. When we study the topology of fiber bundles, the fundamental question is how the topological invariants of the total space, the fiber space, and the base space are related. Even though it is an elementary fact that the Euler characteristic is multiplicative for fiber bundles, for the signature, the same does not hold in general. As the first counterexamples, Atiyah\cite{Atiyah:69} and, independently, Kodaira\cite{Kod:1967}  
provided surface bundles over surfaces with nonvanishing signature. 
In these classical examples, the fiber genus $f$ or the base genus $b$ was fairly big. For example, in Atiyah's example, $f=6$ and $b=129$.
After that, there have been many efforts to find out the smallest possible genera of surface bundle over surface for which the siganture is nonzero.\cite{Endo:1998, BDS:2001, BD:2002, St:2002, EKKOS:2002} 

In the early constructions of surface bundles, the signature of the total space was computed by using the signature formula for ramified coverings created by Hirzebruch\cite{Hi:1969}. However, not all of the bundles can be constructed by using the branched covering method. Instead, in general, the monodromy information of a surface bundle allows us to compute its signature, with the help of Meyer's signature cocycle\cite{Mey:1973} which is a $2$-cocycle of the symplectic group $Sp(2g,\mathbb{R})$. Using the signature cocycle and Birman-Hilden's relations of mapping class group, Meyer proved that if the fiber genus $f\leq 2$ or the base genus $b\leq 1$, then the signature vanishes. Hence, for a nonzero signature, we only need to consider the case when $f\geq3$ and $b\geq2$. He also proved that for every $f\geq 3$ and every $4n\in 4\mathbb{Z}$, there exists a $\Sigma_f$ bundle over $\Sigma_b$ with signature $4n$ for some $b\geq 0$. Based on his result, Endo\cite{Endo:1998} studied the following refined question which is very similiar to Problem $2.18 A$ in Kirby's problem list \cite{Kirby}.

\begin{problem}
For each $f\geq 3$ and each $n \in \mathbb{Z}$, let $b(f,n)$ be the minimal base genus $b$ over which a surface bundle with fiber genus $f$ and signature $4n$ exists. Determine the value $b(f,n)$.
\end{problem}

In \cite{Endo:1998}, Endo showed that $b(f,n)\leq 111|n|$ for any $f\geq 3$.
In \cite{St:2002}, Stipsicz showed that $b(f,2f+2)\leq 4f+20$.
In \cite{EKKOS:2002}, Endo, Kotschick, 
Korkmaz, Ozbagci, and Stipsicz proved that $b(f,n)\leq 8|n|+1$ for any $f\geq 3$ and any $n\neq 0$. In this paper, we improve this upper bound for $b(f,n)$.

\begin{theorem}\label{thm:main}

(a) For every $f\geq3$ and $n\neq 0$, $b(f,n)\leq 7|n|+1$. In particular, there exists a smooth $4$-manifold with signature $4$ which is a $\Sigma_3$-bundle over $\Sigma_8$. \\
(b) For every $f\geq5$ and $n\neq 0$, $b(f,n)\leq 6|n|+1$. In particular, there exists a smooth $4$-manifold with signature $4$ which is a $\Sigma_5$-bundle over $\Sigma_7$. \\
(c) For every $f\geq6$ and $n\neq 0$, $b(f,n)\leq 5|n|+1$. In particular, there exists a smooth $4$-manifold with signature $4$ which is a $\Sigma_6$-bundle over $\Sigma_6$.
\end{theorem}

Our constructions of surface bundles rely on various computations in mapping class groups, which we will introduce in Section $3$. From a geometric point of view, these computations correspond to monodromy factorizations of Lefschetz fibrations. From Lefschetz fibrations, by taking neighborhoods of singular fibers out and gluing them along isomorphic boundaries via fiber-preserving diffeomorphisms, we can construct surface bundles over surfaces. This method was introduced in \cite{EKKOS:2002} to construct a $\Sigma_3$ bundle over $\Sigma_9$ with signature $4$. A key ingredient in this paper is that a clever use of different embeddings of relations in mapping class groups gives rise to more economical, in the sense of small genera, surface bundles with a fixed signature $4$. 

\begin{remark}\cite{Kot:1998}
We may think of $b(f,n)$ as the minimal genus of the surfaces representing the $n$ times generator of $H_2(Mod(\Sigma_f):\mathbb{Z})/Tor$ for fixed $f\geq3$ and $n$.
\end{remark}

On the other hand, the lower bound for $b(f,n)$ was also investigated. Kotschick \cite{Kot:1998} proved $b(f,n)\geq \frac{2|n|}{f-1}+1$, and Hamenstadt\cite{Hamen:2012} proved $b(f,n)\geq \frac{3|n|}{f-1}+1$. Combining the latter with our result, we have $3\leq b(3,1)\leq8$, $2\leq b(5,1)\leq7$, and $2\leq b(6,1)\leq6$.

It is not hard to see that $\frac{b(f,n)}{n}$ converges. Now we define $G_f:=\lim_{n\to\infty} \frac{b(f,n)}{n}$ and improve a priori the upper bound for $G_f$ that appeared in \cite{EKKOS:2002}.

\begin{theorem}\label{thm:cor}
For every odd $f\geq3$, $G_f\leq\frac{14}{f-1}$.
\end{theorem}

\begin{remark}
As far as I know, this is the best known upper bound for $f=3$ or every odd $f$ of the form $3k+1,3k+2$. In fact, for some other $f$'s, better upper bounds are known : for even $f\geq 4$, $G_f\leq\frac{6}{f-2}$\cite{BD:2002}, and for $f=3k\geq6$, $G_{f}\leq\frac{9}{f-2}$\cite{BDS:2001}. 
\end{remark}

\begin{section}{Preliminaries}

\begin{subsection}{Signature}

Let $M$ be a compact oriented topological manifold of dimension $4m$.
Since $M$ is oriented, it admits the fundamental class $[M]\in H_{4m}(M,\partial M)$.
\begin{definition}
The symmetric bilinear form $Q_M:H^{2m}(M,\partial M)\times H^{2m}(M,\partial M)$
$\rightarrow \mathbb{Z}$ defined by $Q_M(a,b):=<a\cup b,[M]>$ is called the intersection form of $M$.
\end{definition}

\begin{remark}
In the smooth case, we can understand $Q_M$ above as the algebraic intersection number of smoothly embedded oriented submanifolds in $M$ representing the Poincar\'e duals of $a$ and $b$.  
\end{remark}

If $a$ or $b$ is a torsion element, then $Q_M$ vanishes, and hence $Q_M$ descends to the cohomology modulo torsion.
\begin{definition}
The signature of $M$, denoted by $\sigma(M)$, is defined to be the signature of the symmetric bilinear form $Q_M$ on $H^{2m}(M,\partial M)/Tor$. If the dimension of $M$ is not divisible by $4$, $\sigma(M)$ is defined to be zero.
\end{definition}   

\end{subsection}

\begin{subsection}{Mapping class group}

Let $\Sigma_g^{r}$ be an oriented surface of genus $g$ with $r$ boundary components and let $\Sigma_g$ be a closed oriented surface of genus $g$. The mapping class group $Mod(\Sigma_g^{r})$ of $\Sigma_g^{r}$ is defined to be the group of isotopy classes of orientation preserving self-homeomorphisms which are identity on each boundary component. Based on the theorem of Dehn, we have a surjective homomorphism $\pi:F(S)\rightarrow Mod(\Sigma_g)$, where $F(S)$ is the free group generated by the generating set $S$ consisting of all the Dehn twists over all isotopy classes of simple closed curves on $\Sigma_g$. Set $R:=Ker\pi$ and call each
word $w$ in the generators $S$ of $Mod(\Sigma_g)$ a relation of $Mod(\Sigma_g)$ if $w\in R$. Now, let us review some famous relations of mapping class groups.
 
Let $a$ and $b$ be two simple closed curves on $\Sigma_g$. If $a$ and $b$ are disjoint, then the supports of the Dehn twists $t_a$ and $t_b$ can be chosen to be disjoint. Hence, there exist commutativity relations $t_a t_b t_a^{-1} t_b^{-1}$ for any disjoint simple closed curves $a$ and $b$. If $a$ intersects $b$ transversely at one point, then there exists a braid relation $t_at_bt_at_b^{-1}t_a^{-1}t_b^{-1}$. It can be derived from more general fact that $ft_af^{-1}=t_{f(a)}$ in $Mod(\Sigma_g)$ for any simple closed curve $a$ on $\Sigma_g$ and any orientation preserving homeomorphism $f$ of $\Sigma_g$. For braid relations, we will take the latter general form $ft_af^{-1}t_{f(a)}^{-1}$. Consider the planar surface $\Sigma_0^{4}$ with boundary components $a,b,c$, and $d$. On the left hand side of Figure $1$, the boundary curves $a,b,c$, and $d$ are in black and the interior curves $x,y,$ and $z$ are in different colors. One can easily check that $t_at_bt_ct_d=t_zt_yt_x$ holds in $Mod(\Sigma_0^{4})$ by applying the Alexander method, and we call $t_d^{-1}t_c^{-1}t_b^{-1}t_a^{-1}t_zt_yt_x$ the lantern relations for 
all embedded subsurfaces $\Sigma_0^{4}\hookrightarrow \Sigma_g$.
For the $k$-chain relations and any other details for mapping class groups, refer to \cite{FM:2012}. One can also deduce the star relations $t_{\delta_3}^{-1}t_{\delta_2}^{-1}t_{\delta_1}^{-1}(t_{\alpha_1}t_{\alpha_2}t_{\alpha_3}t_{\beta})^3$ supported on any embedded subsurfaces $\Sigma_1^{3}\hookrightarrow \Sigma_g$. See Figure $4$ as an example.

We say that two simple closed curves $a$ and $b$ on $\Sigma_g$ are topologically equivalent if there exists a homeomorphism of $\Sigma_g$ sending $a$ to $b$. Similarly, the two collections $\{a_1,\cdots,a_n\}$ and $\{b_1,\cdots,b_n\}$ of simple closed curves on $\Sigma_g$ are called topologically equivalent if there exists a homeomorphism of $\Sigma_g$ sending $a_i$ to $b_i$ simultaneously for all $1\leq i\leq n$. To simplify the notation in the rest of this paper, we will use the notation $w_1^{w_2}$ for the conjugation $w_2^{-1}w_1w_2$.\\

\end{subsection}

\begin{subsection}{Lefschetz fibrations and surface bundles}
\begin{definition}
Let $X$ be a compact oriented $4$-manifold, and $B$ a compact oriented $2$-manifold. A smooth surjective map $f:X\rightarrow B$ is called a Lefschetz fibration if for each critical point $p\in X$ there are local complex coordinates $(z_1,z_2)$ on $X$ around $p$ and $z$ on $B$ around $f(p)$ compatible with the orientations and such that $f(z_1,z_2)=z_1^2+z_2^2$.
\end{definition}
It follows that $f$ has only finitely many critical points, and we may assume that $f$ is injective on the critical set $C=\{p_1,\cdots,p_k\}$. A fiber of $f$ containing a critical point is called a singular fiber, and the genus of $f$ is defined to be the genus of the regular fiber. Notice that if $\nu(f(C))$ denotes an open tubular neighborhood of the set of critical values $f(C)$, then the restriction of $f$ to $f^{-1}(B-\nu(f(C)))$ is a smooth surface bundle over $B-\nu(f(C))$.

For a smooth surface bundle $f:E\rightarrow B$ with a fixed identification $\phi$ of the fiber over the base point $p$ of $B$ with a standard genus $g$ surface $\Sigma_g$, the monodromy representation of $f$ is defined to be an antihomomorphism $\chi:\pi_1(B,p)\rightarrow Mod(\Sigma_g)$ defined as follows. For each loop $l:[0,1]\rightarrow B$, $l^*(E)\rightarrow [0,1]$ is trivial and hence there exists a parametrization $\Phi:[0,1]\times\Sigma_g\rightarrow f^{-1}(l[0,1]))$ with $\Phi|_{0\times\Sigma_g}=\phi^{-1}$. Now define $\chi([l]):=[\Phi|_{0\times\Sigma_g}^{-1}\circ\Phi|_{1\times\Sigma_g}]$.
For the genus $g$ Lefschetz fibration $f:X\rightarrow B$ with a fixed identification of the fiber with $\Sigma_g$, we define the monodromy representation of $f$ to be the monodromy representation of the surface bundle $f:X-f^{-1}(f(C))\rightarrow B-f(C)$. 

A Lefschetz singular fiber can be described by its monodromy. By looking at the local model of the Lefschetz critical point, one can see that the singular fiber is obtained from the regular fiber by collapsing a simple closed curve, called the vanishing cycle. One can also observe that the monodromy along the loop going around one Lefschetz critical value is given by the right-handed Dehn twist along the vanishing cycle. Hence, from the monodromy representation $\chi$ of a Lefschetz fibration, after fixing the generating system $\{a_1,b_1,\cdots,a_h,b_h,l_1,\cdots,l_k\}$ of $\pi_1(B-f(C),p)$, we get the global monodromy $\prod_{i=1}^{h}[\chi(a_i),\chi(b_i)]\prod_{j=1}^{k}t_{\gamma_j}$ since we have $\chi(l_j)=t_{\gamma_j}$ for each $j=1,\cdots,k$; and when $B$ is closed,  $\prod_{i=1}^{h}[\chi(a_i),\chi(b_i)]\prod_{j=1}^{k}t_{\gamma_j}$ $=1$ in $Mod(\Sigma_g)$, and this is called the monodromy factorization of a Lefschetz fibration. Conversely, a factorization $\prod_{i=1}^{h}[\alpha_i,\beta_i]\prod_{j=1}^{k}t_{\gamma_j}=1$ of identity in $Mod(\Sigma_g)$ gives rise to a genus $g$ Lefschetz fibration over $\Sigma_h$. For this, first observe that a product $\prod_{i=1}^{h}[\alpha_i,\beta_i]$ of $h$ commutators in $Mod(\Sigma_g)$ gives a $\Sigma_g$ bundle over $\Sigma_h^{1}$. Also, a product $\prod_{j=1}^{k}t_j$ of right-handed Dehn twists $t_j$ in $Mod(\Sigma_g)$ gives a genus $g$ Lefshetz fibration over $D^2$. By combining these two constructions, a word $w=\prod_{i=1}^{h}[\alpha_i,\beta_i]\prod_{j=1}^{k}t_j\in Mod(\Sigma_g)$ gives the genus $g$ Lefschetz fibration over $\Sigma_h^{1}$, and if $w=1$ in $Mod(\Sigma_g)$ we can close up to a Lefschetz fibration over $\Sigma_h$.

Two Lefschetz fibrations $f_1:X_1\rightarrow B_1, f_2:X_2\rightarrow B_2$ are called isomorphic if there exist orientation preserving diffeomorphisms $H:X_1\rightarrow X_2$ and $h:B_1\rightarrow B_2$ such that $f_2\circ H=h\circ f_1$. The isomorphism class of a Lefschetz fibration is determined by an equivalence class of its monodromy representation. Oriented genus $g$ surface bundles over surfaces of genus $h$ are classified, up to isomorphism, by homotopy classes of the classifying map $\Sigma_h\rightarrow BDiff^+\Sigma_g$ since the structure group is $Diff^+\Sigma_g$. If $g\geq2$, then according to the Earle-Eells theorem and the $K(\pi,1)$ theory, they are classified by the conjugacy classes of the induced homomorphisms $\pi_1(\Sigma_h)\rightarrow Mod(\Sigma_g)$. Therefore, $\prod_{i=1}^{h}[\alpha_i,\beta_i]=1$ in $Mod(\Sigma_g)$ , up to global conjugations, determines the genus $g$ surface bundle over a surface of genus $h$.
\end{subsection}
\end{section}

%%%%%%%%%%%%%%%%%%%%%%%%%%%%%%%%%%%%%%%%%%%%%%%%%%
%
% section 3 : subtraction of Lefschetz fibrations
%
%%%%%%%%%%%%%%%%%%%%%%%%%%%%%%%%%%%%%%%%%%%%%%%%%%%%

\begin{section}{Subtraction of Lefschetz fibrations}

In the study of manifold theory, a common way to construct a new manifold from a given manifold is a cut-and-paste operation. 
To construct a new $4$-manifold which is a surface bundle over a surface, 
H.~Endo, M.~Korkmaz, D.~Kotschick, B.~Ozbagci and A.~Stipsicz 
introduced an operation, called the "subtraction of Lefschetz fibrations", in \cite{EKKOS:2002}. Let us first explain it here in a generalized version.

Let $f:X\rightarrow B_1$ be a Lefschetz fibration with $m$ critical values $q_1^{(1)},\cdots,q_m^{(1)}$ and let $g:Y\rightarrow B_2$ be another Lefschetz fibration with $k\leq m$ critical values $q_1^{(2)},\cdots,q_k^{(2)}$. 
Suppose that 
$f:f^{-1}(D_1)\rightarrow D_1$ and $g:g^{-1}(D_2)\rightarrow D_2$ are isomorphic where $D_1\subset B_1$ is a disk containing some critical values $q_1^{(1)},\cdots, q_k^{(1)}$
and $D_2\subset B_2$ is a disk containing $q_1^{(2)},\cdots,q_k^{(2)}$.
Then, the manifolds $X\backslash f^{-1}(D_1)$ and $Y\backslash g^{-1}(D_2)$ have a diffeomorphic boundary, and after reversing the orientation of one of them, this diffeomorphism can be chosen to be fiber-preserving and orientation-reversing. Fix such a diffeomorphism $\phi$ and then glue $\overline{Y\backslash 
g^{-1}(D_2)}$, the manifold $Y\backslash g^{-1}(D_2)$ with the reversed orientation, to $X\backslash f^{-1}(D_1)$ using this diffeomorphism $\phi$. Note that the resulting manifold, denoted by $X-Y$, inherits a natural orientation and admits a smooth fibration $f\cup g : X\backslash f^{-1}(D_1) \cup
\overline{Y\backslash g^{-1}(D_2)} \rightarrow B_1\#B_2$. This is 
a Lefschetz fibration with $m-k$ singular fibers.  
In particular, for $k=m$, we get a surface bundle over a surface. 
In general, after repeatedly subtracting Lefschetz fibrations, we get $X-Y_1-Y_2-\cdots-Y_n$, a surface bundle over a surface, under the following assumptions. Let $f:X\rightarrow B_0$ be a Lefschetz fibration with $m$ critical values $\{q_{1,1}^{(0)},\cdots ,q_{1,k_1}^{(0)},q_{2,1}^{(0)},\cdots ,q_{2,k_2}^{(0)},\cdots ,q_{n,1}^{(0)},\cdots ,q_{n,k_n}^{(0)}\}$ and 
$g_1:Y_1\rightarrow B_1$, $\cdots$ , $g_n:Y_n\rightarrow B_n$
be Lefschetz fibrations with critical values
$\{q_1^{(1)},\cdots,q_{k_1}^{(1)}\}$, $\cdots$ ,
$\{q_1^{(n)},\cdots,q_{k_n}^{(n)}\}$, respectively. 
We assume that $k_1+\cdots +k_n = m$ and that
$f:f^{-1}(D_{0,i})\rightarrow D_{0,i}$ is isomorphic to $g_i:g_i^{-1}(D_i)\rightarrow D_i$ for each $1\leq i\leq n$,
where each $D_{0,i}\subset B_0$ is a disk containing 
$q_1^{(0)},\cdots q_{k_i}^{(0)}$ and $D_i \subset B_i$ is a disk containing $q^{(i)}_1,\cdots,q_{k_i}^{(i)}$.

%%%%%%%%%%%%%%%%%%%%%%%%%%%%%%%%%%%%%%%%%%%%
%
% prop [EKKOS]
%
%%%%%%%%%%%%%%%%%%%%%%%%%%%%%%%%%%%%%%%%%%%%%%

In order to use the subtraction method explained above, we need
to construct the building blocks $X$ and $Y_i$'s. First, we describe various gluing pieces $Y_i$.

\begin{proposition}\cite{EKKOS:2002}
Let $f\geq3$ and let $a$ be a simple closed curve on $\Sigma_f$.
In the mapping class group $Mod(\Sigma_f)$,\\
\noindent(a) $t_a^2$ can be written as a product of two commutators,\\
(b) if $a$ is nonseparating, then $t_a^4$ can be written as a product 
of three commutators. 
\end{proposition}

\begin{remark}
This proposition gives us two genus $f\geq 3$ Lefschetz fibrations $Y_1\rightarrow\Sigma_2$ and $Y_2\rightarrow\Sigma_3$ whose monodromy factorizations are given by $[f_1,g_1][f_2,$ $g_2]t_a^2=1$ and $[f_3,g_3][f_4,g_4][f_5,g_5]t_a^4=1$ for some mapping classes $f_i,g_i\in Mod(\Sigma_f)$ for $1\leq i\leq 5$. Generally, for every $n$, we can obtain a Lefschetz fibration which has $n$ singular fibers and the monodromy $t_a^n$ using a daisy relation. 
\end{remark}

%%%%%%%%%%%%%%%%%%%%%%%%%%%%%%%%%%%%%%%%%%%%%%%%%%
%
% figure 1 : four lanterns on a genus 5 surface
%
%%%%%%%%%%%%%%%%%%%%%%%%%%%%%%%%%%%%%%%%%%%%%%%%%%%%%%

\begin{figure}

\definecolor{qqqqff}{rgb}{0.0,0.0,1.0}
\definecolor{ffqqqq}{rgb}{1.0,0.0,0.0}
\definecolor{qqzzqq}{rgb}{0.0,0.6,0.0}

\begin{tikzpicture}[line cap=round,line join=round,>=triangle 45,x=1.0cm,y=1.0cm,scale=0.6]
\clip(-5.46,-3.24) rectangle (17.82,6.9);
\draw(1.1,1.66) circle (3.5cm);
\draw (3.72,-0.84) node[anchor=north west] {$b$};
\draw(-0.44,2.62) circle (1.0cm);
\draw(2.56,2.72) circle (1.0cm);
\draw(1.22,0.24) circle (1.0cm);
\draw (1.04,1.1) node[anchor=north west] {$c$};
\draw (2.36,3.7600000000000002) node[anchor=north west] {$a$};
\draw (-0.66,3.7600000000000002) node[anchor=north west] {$d$};
\draw [color=qqzzqq] (-0.88,2.6) circle (0.3cm);
\draw [color=qqzzqq] (0.02,2.62) circle (0.3cm);
\draw [color=ffqqqq] (2.02,2.68) circle (0.3cm);
\draw [color=ffqqqq] (2.98,2.66) circle (0.3cm);
\draw [color=qqqqff] (0.82,0.1) circle (0.3cm);
\draw [color=qqqqff] (1.68,0.14) circle (0.3cm);
\draw [color=ffqqqq](1.6600000000000001,3.06) node[anchor=north west] {$a_2$};
\draw [color=ffqqqq](2.6,3.06) node[anchor=north west] {$a_3$};
\draw [rotate around={60.67698520366726:(1.84,1.3499999999999999)},color=ffqqqq] (1.84,1.3499999999999999) ellipse (3.095531607390996cm and 1.5223061230766568cm);
\draw [color=ffqqqq](3.18,0.5) node[anchor=north west] {$x$};
\draw [rotate around={-54.37804256251799:(0.49031058991969545,1.279092341422993)},color=qqzzqq] (0.49031058991969545,1.279092341422993) ellipse (2.9471616525075603cm and 1.464227544933488cm);
\draw [color=qqzzqq](-1.46,0.78) node[anchor=north west] {$y$};
\draw [color=qqzzqq](-1.32,3.08) node[anchor=north west] {$d_1$};
\draw [color=qqzzqq](-0.36,3.08) node[anchor=north west] {$d_2$};
\draw [rotate around={0.24612781462116637:(1.0301126819927966,2.684146493214845)},color=qqqqff] (1.0301126819927966,2.684146493214845) ellipse (3.0032646659857236cm and 1.3923356829293465cm);
\draw [color=qqqqff](0.72,4.96) node[anchor=north west] {$z$};
\draw [color=qqqqff](0.46,0.52) node[anchor=north west] {$c_2$};
\draw [color=qqqqff](1.3,0.54) node[anchor=north west] {$c_1$};
\draw [color=qqqqff] (1.1,1.66) circle (3.4cm);
\draw [color=qqzzqq] (1.22,0.24) circle (1.1cm);
\draw [color=ffqqqq] (1.22,0.24) circle (1.2cm);
\draw [shift={(9.62,2.18)}] plot[domain=3.750281960917206:5.624379271062111,variable=\t]({1.0*0.8044874144447495*cos(\t r)+-0.0*0.8044874144447495*sin(\t r)},{0.0*0.8044874144447495*cos(\t r)+1.0*0.8044874144447495*sin(\t r)});
\draw [shift={(9.64,0.9)}] plot[domain=0.9964914966201953:2.198104519070659,variable=\t]({1.0*0.8099382692526632*cos(\t r)+-0.0*0.8099382692526632*sin(\t r)},{0.0*0.8099382692526632*cos(\t r)+1.0*0.8099382692526632*sin(\t r)});
\draw [rotate around={-1.0369144645300876:(9.569999999999997,1.6000000000000005)}] (9.569999999999997,1.6000000000000005) ellipse (4.030225503474254cm and 3.3700174493397355cm);
\draw [shift={(9.64,3.92)}] plot[domain=3.7502819609172047:5.624379271062109,variable=\t]({1.0*0.8044874144447508*cos(\t r)+-0.0*0.8044874144447508*sin(\t r)},{0.0*0.8044874144447508*cos(\t r)+1.0*0.8044874144447508*sin(\t r)});
\draw [shift={(9.64,2.66)}] plot[domain=0.9964914966201952:2.198104519070659,variable=\t]({1.0*0.809938269252663*cos(\t r)+-0.0*0.809938269252663*sin(\t r)},{0.0*0.809938269252663*cos(\t r)+1.0*0.809938269252663*sin(\t r)});
\draw [shift={(9.64,0.5)}] plot[domain=3.7502819609172047:5.624379271062109,variable=\t]({1.0*0.8044874144447508*cos(\t r)+-0.0*0.8044874144447508*sin(\t r)},{0.0*0.8044874144447508*cos(\t r)+1.0*0.8044874144447508*sin(\t r)});
\draw [shift={(9.66,-0.78)}] plot[domain=0.9964914966201953:2.198104519070659,variable=\t]({1.0*0.8099382692526632*cos(\t r)+-0.0*0.8099382692526632*sin(\t r)},{0.0*0.8099382692526632*cos(\t r)+1.0*0.8099382692526632*sin(\t r)});
\draw [shift={(7.34,2.2)}] plot[domain=3.7502819609172047:5.62437927106211,variable=\t]({1.0*0.8044874144447509*cos(\t r)+-0.0*0.8044874144447509*sin(\t r)},{0.0*0.8044874144447509*cos(\t r)+1.0*0.8044874144447509*sin(\t r)});
\draw [shift={(7.36,0.92)}] plot[domain=0.9964914966201953:2.198104519070659,variable=\t]({1.0*0.8099382692526632*cos(\t r)+-0.0*0.8099382692526632*sin(\t r)},{0.0*0.8099382692526632*cos(\t r)+1.0*0.8099382692526632*sin(\t r)});
\draw [shift={(11.9,2.2)}] plot[domain=3.7502819609172047:5.624379271062109,variable=\t]({1.0*0.8044874144447509*cos(\t r)+-0.0*0.8044874144447509*sin(\t r)},{0.0*0.8044874144447509*cos(\t r)+1.0*0.8044874144447509*sin(\t r)});
\draw [shift={(11.92,0.92)}] plot[domain=0.9964914966201953:2.198104519070659,variable=\t]({1.0*0.8099382692526632*cos(\t r)+-0.0*0.8099382692526632*sin(\t r)},{0.0*0.8099382692526632*cos(\t r)+1.0*0.8099382692526632*sin(\t r)});
\draw [color=qqqqff](8.56,4.6000000000000005) node[anchor=north west] {$a_2$};
\draw [shift={(11.1,-1.12)},dash pattern=on 2pt off 2pt,color=qqqqff]  plot[domain=2.6679951241961817:3.5251008932857952,variable=\t]({1.0*1.7978876494375278*cos(\t r)+-0.0*1.7978876494375278*sin(\t r)},{0.0*1.7978876494375278*cos(\t r)+1.0*1.7978876494375278*sin(\t r)});
\draw [shift={(8.06,-0.98)},color=qqqqff]  plot[domain=-0.5425793241259909:0.4324077755705381,variable=\t]({1.0*1.5880806024884249*cos(\t r)+-0.0*1.5880806024884249*sin(\t r)},{0.0*1.5880806024884249*cos(\t r)+1.0*1.5880806024884249*sin(\t r)});
\draw [color=qqqqff](9.8,-0.64) node[anchor=north west] {$c_1$};
\draw [shift={(6.12,0.14)},color=qqqqff]  plot[domain=1.0815415982956988:1.9313623272730025,variable=\t]({1.0*1.588080602488425*cos(\t r)+-0.0*1.588080602488425*sin(\t r)},{0.0*1.588080602488425*cos(\t r)+1.0*1.588080602488425*sin(\t r)});
\draw [shift={(6.22,2.92)},dash pattern=on 2pt off 2pt,color=qqqqff]  plot[domain=4.230450728427064:5.138574150327251,variable=\t]({1.0*1.4671059948074643*cos(\t r)+-0.0*1.4671059948074643*sin(\t r)},{0.0*1.4671059948074643*cos(\t r)+1.0*1.4671059948074643*sin(\t r)});
\draw [color=qqqqff](5.94,2.58) node[anchor=north west] {$c_2$};
\draw [shift={(8.42,0.08)},color=qqqqff]  plot[domain=1.081541598295699:1.9313623272730025,variable=\t]({1.0*1.5880806024884242*cos(\t r)+-0.0*1.5880806024884242*sin(\t r)},{0.0*1.5880806024884242*cos(\t r)+1.0*1.5880806024884242*sin(\t r)});
\draw [shift={(8.52,2.84)},dash pattern=on 2pt off 2pt,color=qqqqff]  plot[domain=4.230450728427064:5.13857415032725,variable=\t]({1.0*1.4671059948074643*cos(\t r)+-0.0*1.4671059948074643*sin(\t r)},{0.0*1.4671059948074643*cos(\t r)+1.0*1.4671059948074643*sin(\t r)});
\draw [color=qqqqff](8.1,2.64) node[anchor=north west] {$d_1$};
\draw [shift={(8.06,2.42)},dash pattern=on 2pt off 2pt,color=qqqqff]  plot[domain=-0.4145068745847862:0.4315730512235055,variable=\t]({1.0*1.6387800340497196*cos(\t r)+-0.0*1.6387800340497196*sin(\t r)},{0.0*1.6387800340497196*cos(\t r)+1.0*1.6387800340497196*sin(\t r)});
\draw [shift={(11.22,2.38)},color=qqqqff]  plot[domain=2.73670086730471:3.5251008932857943,variable=\t]({1.0*1.8277855454073395*cos(\t r)+-0.0*1.8277855454073395*sin(\t r)},{0.0*1.8277855454073395*cos(\t r)+1.0*1.8277855454073395*sin(\t r)});
\draw [shift={(11.08,4.16)},color=qqqqff]  plot[domain=2.66747422308429:3.568220146716669,variable=\t]({1.0*1.7084495895401768*cos(\t r)+-0.0*1.7084495895401768*sin(\t r)},{0.0*1.7084495895401768*cos(\t r)+1.0*1.7084495895401768*sin(\t r)});
\draw [shift={(8.1,4.26)},dash pattern=on 2pt off 2pt,color=qqqqff]  plot[domain=-0.4691120354852938:0.4315730512235054,variable=\t]({1.0*1.6368261972488107*cos(\t r)+-0.0*1.6368261972488107*sin(\t r)},{0.0*1.6368261972488107*cos(\t r)+1.0*1.6368261972488107*sin(\t r)});
\draw [color=qqqqff](8.620000000000001,2.96) node[anchor=north west] {$d_2$};
\draw [shift={(9.98,1.76)},color=qqqqff]  plot[domain=4.6321021651016725:6.085789747329706,variable=\t]({1.0*1.7456230979223435*cos(\t r)+-0.0*1.7456230979223435*sin(\t r)},{0.0*1.7456230979223435*cos(\t r)+1.0*1.7456230979223435*sin(\t r)});
\draw [shift={(9.4,2.48)},dash pattern=on 2pt off 2pt,color=qqqqff]  plot[domain=4.89724601240313:5.84334272436385,variable=\t]({1.0*2.502638607550039*cos(\t r)+-0.0*2.502638607550039*sin(\t r)},{0.0*2.502638607550039*cos(\t r)+1.0*2.502638607550039*sin(\t r)});
\draw [color=qqqqff](11.200000000000001,0.44) node[anchor=north west] {$b$};
\draw [shift={(10.88,3.12)},color=qqqqff]  plot[domain=-0.9169266821669524:0.4513951178302744,variable=\t]({1.0*1.7599527332175413*cos(\t r)+-0.0*1.7599527332175413*sin(\t r)},{0.0*1.7599527332175413*cos(\t r)+1.0*1.7599527332175413*sin(\t r)});
\draw [color=qqqqff](12.620000000000001,2.54) node[anchor=north west] {$a_3$};
\draw [shift={(9.84,3.34)},dash pattern=on 2pt off 2pt,color=qqqqff]  plot[domain=-0.6416796073957673:0.2014550774686881,variable=\t]({1.0*2.6741364240128105*cos(\t r)+-0.0*2.6741364240128105*sin(\t r)},{0.0*2.6741364240128105*cos(\t r)+1.0*2.6741364240128105*sin(\t r)});
\end{tikzpicture}
\caption{Supports of four lantern relations and an embedding of $\Sigma_0^{7}$ into a genus $5$ surface}
\end{figure}
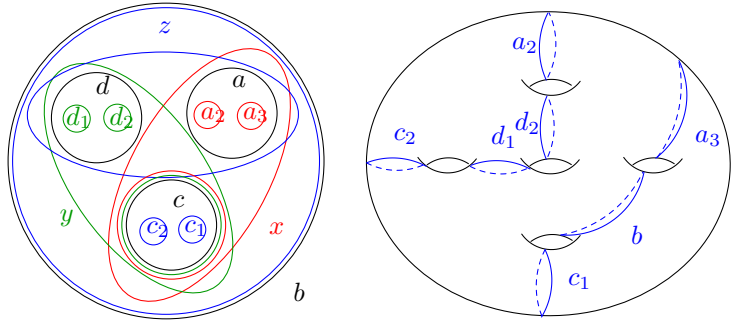

%%%%%%%%%%%%%%%%%%%%%%%%%%%%%%%%%%%%%%%%%%%%%%%%%%%%
%
% prop : b^2c^2=[,][,][,]
%
%%%%%%%%%%%%%%%%%%%%%%%%%%%%%%%%%%%%%%%%%%%%

The following two propositions allow us to glue building blocks along more complicated monodromies in the sense that they are products of Dehn twists along distinct simple closed curves.

\begin{proposition}
Let $f\geq5$ and let $b$,$c$ be disjoint simple closed curves on 
$\Sigma_f$ such that $\Sigma_f-b-c$ is connected.
In $Mod(\Sigma_{f})$,
$t_b^2 t_c^2$ can be written as a product of three commutators.

\begin{proof}
We may assume $b$ and $c$ are embedded, as shown in Figure $1$, because any pair of simple closed curves whose complement in $\Sigma_f$ is connected is topologically equivalent. On the planar surface $\Sigma_0^{7}$ in Figure $1$, the following four lantern relations hold. $L_1:=t_{a}^{-1}t_{b}^{-1}t_{c}^{-1}t_{d}^{-1}t_{y}t_{x}t_{z}$ ,   
 $L_2:=t_{d}t_{D_2}t_{D_1}t_{d_1}^{-1}t_{d_2}^{-1}t_{c}^{-1}t_{y}^{-1}$,
 $L_3:=t_{x}^{-1}t_{a_2}^{-1}t_{a_3}^{-1}t_{c}^{-1}t_{a}t_{A_3}t_{A_2}$,
 $L_4:=t_{z}^{-1}t_{c_1}^{-1}t_{c_2}^{-1}t_{b}^{-1}t_{c}t_{C_2}t_{C_1}$ . Here, $D_1$ is an interior curve surrounding two boundary curves except $d_1$, and all other curves denoted by capital letters are defined similarly.
After embedding $\Sigma_0^{7}$ into $\Sigma_f$ with $f\geq5$,
as shown in Figure $1$, we have $1=L_1\cdot L_2^{t_yt_xt_z}\cdot L_3^{t_z}
\cdot L_4 = t_b^{-1}t_c^{-1}t_{D_2}t_{d_2}^{-1}t_{D_1}t_{d_1}^{-1}t_{c}^{-1}t_{A_3}t_{a_3}^{-1}t_{A_2}t_{a_2}^{-1}$
$t_{b}^{-1}t_{C_2}t_{c_2}^{-1}t_{C_1}t_{c_1}^{-1}$
in $Mod(\Sigma_f)$. Since both $\Sigma_f -D_2-d_2$ and $\Sigma_f -D_1-d_1$ are connected, $\{d_2,D_2\}$ and $\{D_1,d_1\}$ are topologically equivalent and then $t_{D_2}t_{d_2}^{-1}t_{D_1}t_{d_1}^{-1} = [t_{D_2}t_{d_2}^{-1},\phi_1]$ for some $\phi_1 \in Mod(\Sigma_f)$. Similarly, $t_{A_3}t_{a_3}^{-1}t_{A_2}t_{a_2}^{-1}=[t_{A_3}t_{a_3}^{-1},\phi_2]$ and $t_{C_2}t_{c_2}^{-1}t_{C_1}t_{c_1}^{-1}=[t_{C_2}t_{c_2}^{-1},\phi_3]$ for some
$\phi_2, \phi_3 \in Mod(\Sigma_f)$. \\
Therefore, $t_b^{2}t_c^{2}=[t_{D_2}t_{d_2}^{-1},\phi_1]^{(t_bt_c)^{-1}}[t_{A_3}t_{a_3}^{-1},\phi_2]^{t_b^{-1}}[t_{C_2}t_{c_2}^{-1},\phi_3]$.
\end{proof}
\end{proposition}

%%%%%%%%%%%%%%%%%%%%%%%%%%%%%%%%%%%%%%%%%%%%%%%%%%
%
% figure 2 : two lanterns on a genus 6 surface
%
%%%%%%%%%%%%%%%%%%%%%%%%%%%%%%%%%%%%%%%%%%%%%%%%%%%%
\begin{figure}
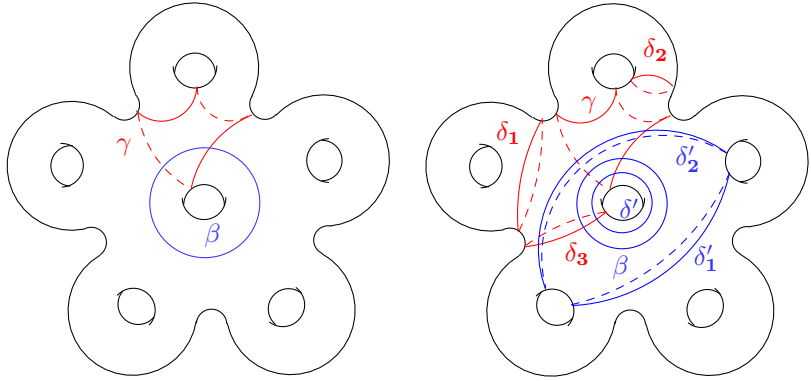


\definecolor{qqqqff}{rgb}{0.0,0.0,1.0}
\definecolor{ttttff}{rgb}{0.2,0.2,1.0}
\definecolor{ffqqqq}{rgb}{1.0,0.0,0.0}


\caption{Supports of two lantern relations embedded in a genus $6$ surface}
\end{figure}

\begin{figure}
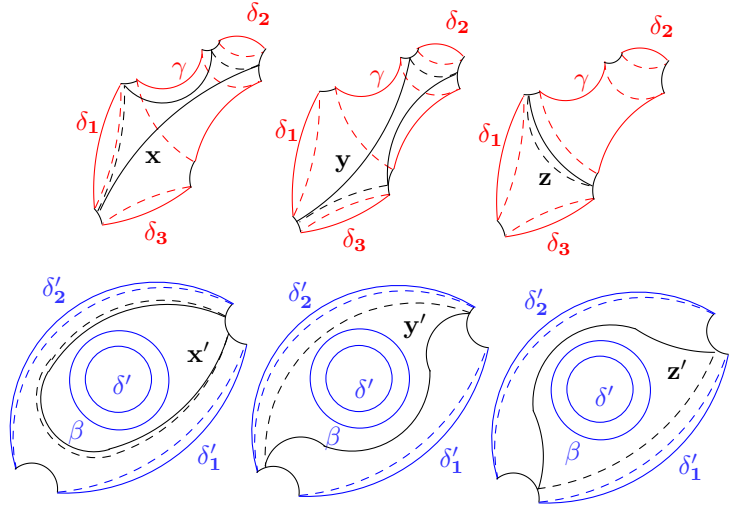

\definecolor{ffqqqq}{rgb}{1.0,0.0,0.0}
\definecolor{qqqqff}{rgb}{0.0,0.0,1.0}
\definecolor{ttttff}{rgb}{0.2,0.2,1.0}
%
\caption{Interior curves for two lantern relations}

\end{figure}

%%%%%%%%%%%%%%%%%%%%%%%%%%
%
% prop :\beta\gamma=[,][,][,]
%
%%%%%%%%%%%%%%%%%%%%%%%%%%%%%55

\begin{proposition}
Let $f\geq6$ and let $\beta$,$\gamma$ be simple closed curves 
on $\Sigma_f$ embedded, as shown in Figure $2$. In $Mod(\Sigma_f)$,
$t_\beta t_\gamma$ can be written as a product of three commutators. 

\begin{proof}
Choose two lantern relations with their supports on $\Sigma_{f}$, as shown in Figure $2$: $L_1:=t_{\gamma}^{-1}t_{\delta_1}^{-1}t_{\delta_2}^{-1}t_{\delta_3}^{-1}t_yt_xt_z$ 
and $L_2:=t_{x'}t_{z'}t_{y'}t_{\delta'}^{-1}t_{\delta_1'}^{-1}t_{\delta_2'}^{-1}t_{\beta}^{-1}$.
For interior curves, see Figure $3$.
It follows that $1=L_1\cdot L_2=t_{\gamma}^{-1}t_{\delta_2}^{-1}t_yt_{\delta_3}^{-1}t_xt_zt_{\delta_1}^{-1}
\cdot t_{x'}t_{\delta_2'}^{-1}t_{z'}t_{\delta'}^{-1}t_{y'}t_{\delta_1'}^{-1}t_{\beta}^{-1}$.
In Figure $2$ and Figure $3$, we can see that $\delta_1$ and $x'$ are separating curves on $\Sigma_f$ and that both $\Sigma_f-z-\delta_1$ and $\Sigma_f-\delta_2'-x'$ are homeomorphic to $\Sigma_1^{1}\cup\Sigma_{f-2}^{3}$. Hence, we have
$t_zt_{\delta_1}^{-1}t_{x'}t_{\delta_2'}^{-1} = [t_zt_{\delta_1}^{-1},\phi_2]$ for some $\phi_2$.
Similarly, we have
$t_{\delta_2}^{-1}t_yt_{\delta_3}^{-1}t_x=[t_{\delta_2}^{-1}t_y,\phi_1]$ and
$t_{z'}t_{\delta'}^{-1}t_{y'}t_{\delta_1'}^{-1}=[t_{z'}t_{\delta'}^{-1},\phi_3]$ for some $\phi_1$ and $\phi_3$.\\
Therefore, $t_\beta t_\gamma=[t_{\delta_2}^{-1}t_y,\phi_1]^{t_{\beta}^{-1}}[t_zt_{\delta_1}^{-1},\phi_2]^{t_{\beta}^{-1}}[t_{z'}t_{\delta'}^{-1},\phi_3]^{t_{\beta}^{-1}}$.
\end{proof}

\end{proposition}

In Proposition $11$ of \cite{EKKOS:2002}, they constructed a genus $f\geq 3$ Lefschetz fibration over a torus with $10$ singular fibers using a two-holed torus relation which is also called a $3$-chain relation. In the following three Propositions, we generalize 
this construction of a Lefschetz fibration.

%%%%%%%%%%%%%%%%%%%%%%%%%%%%%%%%%%%%%%%%%%%%%%%%%%%%%%%%%
% 
% figure 4 : star relation
% 
%%%%%%%%%%%%%%%%%%%%%%%%%%%%%%%%%%%%%%%%%%%%%%%%%%%%%%%%%%

\begin{figure}
\definecolor{qqqqff}{rgb}{0.0,0.0,1.0}
\begin{tikzpicture}[line cap=round,line join=round,>=triangle 45,x=1.0cm,y=1.0cm,scale=0.5]
\clip(-2.8226558265582624,-8.070894308943288) rectangle (8.494417344173444,-0.09257452574546186);
\draw [shift={(2.0316931165311676,-1.278454200542006)},color=qqqqff]  plot[domain=-0.5286057039102916:0.46642152699005807,variable=\t]({1.0*1.4276371146451883*cos(\t r)+-0.0*1.4276371146451883*sin(\t r)},{0.0*1.4276371146451883*cos(\t r)+1.0*1.4276371146451883*sin(\t r)});
\draw [shift={(4.394728346883469,-1.282573441734417)},dotted,color=qqqqff]  plot[domain=2.6802869606822184:3.6981919716000196,variable=\t]({1.0*1.3413826813843628*cos(\t r)+-0.0*1.3413826813843628*sin(\t r)},{0.0*1.3413826813843628*cos(\t r)+1.0*1.3413826813843628*sin(\t r)});
\draw [shift={(2.025817777777778,-4.039646612466124)},color=qqqqff]  plot[domain=-0.5286057039102916:0.4664215269900509,variable=\t]({1.0*1.4276371146451885*cos(\t r)+-0.0*1.4276371146451885*sin(\t r)},{0.0*1.4276371146451885*cos(\t r)+1.0*1.4276371146451885*sin(\t r)});
\draw [shift={(4.422283902439028,-4.022085636856377)},dotted,color=qqqqff]  plot[domain=2.680286960682219:3.698191971600016,variable=\t]({1.0*1.3413826813843654*cos(\t r)+-0.0*1.3413826813843654*sin(\t r)},{0.0*1.3413826813843654*cos(\t r)+1.0*1.3413826813843654*sin(\t r)});
\draw [shift={(2.079107750677507,-6.809077506775067)},color=qqqqff]  plot[domain=-0.5286057039102925:0.46642152699005685,variable=\t]({1.0*1.4276371146451867*cos(\t r)+-0.0*1.4276371146451867*sin(\t r)},{0.0*1.4276371146451867*cos(\t r)+1.0*1.4276371146451867*sin(\t r)});
\draw [shift={(4.469698536585368,-6.8131967479674955)},dotted,color=qqqqff]  plot[domain=2.6802869606822064:3.698191971600017,variable=\t]({1.0*1.3413826813843732*cos(\t r)+-0.0*1.3413826813843732*sin(\t r)},{0.0*1.3413826813843732*cos(\t r)+1.0*1.3413826813843732*sin(\t r)});
\draw(3.284332682926831,-2.7091317073170873) circle (0.7cm);
\draw(3.3475521951219513,-5.431082926829268) circle (0.7cm);
\draw [rotate around={0.8962503768615129:(2.7894617344173454,-4.095270460704711)}] (2.7894617344173454,-4.095270460704711) ellipse (4.796658876222519cm and 3.4879532811608778cm);
\draw [shift={(0.534283902439026,-4.184253658536672)}] plot[domain=0.02524715954416194:3.1668398131339552,variable=\t]({1.0*0.6956363282989587*cos(\t r)+-0.0*0.6956363282989587*sin(\t r)},{0.0*0.6956363282989587*cos(\t r)+1.0*0.6956363282989587*sin(\t r)});
\draw [shift={(0.5342839024390253,-3.89449756097561)}] plot[domain=-3.106884375364518:0.034708278225275,variable=\t]({1.0*0.759091325658558*cos(\t r)+-0.0*0.759091325658558*sin(\t r)},{0.0*0.759091325658558*cos(\t r)+1.0*0.759091325658558*sin(\t r)});
\draw [rotate around={0.0:(0.5421863414634164,-4.096448780487974)},color=qqqqff] (0.5421863414634164,-4.096448780487974) ellipse (1.174240923106269cm and 1.0399635736371289cm);
\draw [shift={(2.705799457994582,-4.298536585366058)},color=qqqqff]  plot[domain=1.581665463972167:2.7509856098921084,variable=\t]({1.0*1.9946977695846446*cos(\t r)+-0.0*1.9946977695846446*sin(\t r)},{0.0*1.9946977695846446*cos(\t r)+1.0*1.9946977695846446*sin(\t r)});
\draw [shift={(0.797940379403796,-1.5017886178863873)},dotted,color=qqqqff]  plot[domain=4.754916515004474:5.85017944555596,variable=\t]({1.0*2.0397846697173474*cos(\t r)+-0.0*2.0397846697173474*sin(\t r)},{0.0*2.0397846697173474*cos(\t r)+1.0*2.0397846697173474*sin(\t r)});
\draw [shift={(2.467317073170734,-3.9299728997292065)},color=qqqqff]  plot[domain=3.541833310623767:4.826829285654753,variable=\t]({1.0*1.8361748570397534*cos(\t r)+-0.0*1.8361748570397534*sin(\t r)},{0.0*1.8361748570397534*cos(\t r)+1.0*1.8361748570397534*sin(\t r)});
\draw [shift={(0.7762601626016282,-6.900162601626227)},dotted,color=qqqqff]  plot[domain=0.5719181998612627:1.5930148921216158,variable=\t]({1.0*2.2431429792381357*cos(\t r)+-0.0*2.2431429792381357*sin(\t r)},{0.0*2.2431429792381357*cos(\t r)+1.0*2.2431429792381357*sin(\t r)});
\draw [shift={(-1.0448780487804858,-5.490948509485304)},color=qqqqff]  plot[domain=0.8798530987924942:2.279422598922568,variable=\t]({1.0*1.4628926098937516*cos(\t r)+-0.0*1.4628926098937516*sin(\t r)},{0.0*1.4628926098937516*cos(\t r)+1.0*1.4628926098937516*sin(\t r)});
\draw [shift={(-1.0448780487804858,-3.0627642276424853)},dotted,color=qqqqff]  plot[domain=4.098369943376749:5.359221458088224,variable=\t]({1.0*1.6180473198521663*cos(\t r)+-0.0*1.6180473198521663*sin(\t r)},{0.0*1.6180473198521663*cos(\t r)+1.0*1.6180473198521663*sin(\t r)});
\draw [color=qqqqff](-0.004227642276420151,-2.087154471544918) node[anchor=north west] {$\mathbf{\beta}$};
\draw [color=qqqqff](3.768130081300815,-0.98146341463435) node[anchor=north west] {$\mathbf{\delta_1}$};
\draw [color=qqqqff](3.833170731707319,-6.3364769647698465) node[anchor=north west] {$\mathbf{\delta_2}$};
\draw [color=qqqqff](3.833170731707319,-3.604769647696678) node[anchor=north west] {$\mathbf{\delta_3}$};
\draw [color=qqqqff](-1.6085636856368535,-3.101246612466326) node[anchor=north west] {$\mathbf{\alpha_1}$};
\draw [color=qqqqff](1.0797831978319807,-5.599349593496135) node[anchor=north west] {$\mathbf{\alpha_2}$};
\draw [color=qqqqff](1.470027100271005,-2.6725203252034544) node[anchor=north west] {$\mathbf{\alpha_3}$};
\draw (5.404227642276425,-3.322926829268494) node[anchor=north west] {$\mathbf{\ldots}$};
\draw (5.129105691056912,-4.66710027100291) node[anchor=north west] {$f-3$};
\draw (5.29867208672088,-4.086612466124862) node[anchor=north west] {$ \underbrace{}$};
\draw [shift={(6.651598915989159,-3.778211382114026)}] plot[domain=3.5124839424024565:5.608444364956036,variable=\t]({1.0*0.4187145347604966*cos(\t r)+-0.0*0.4187145347604966*sin(\t r)},{0.0*0.4187145347604966*cos(\t r)+1.0*0.4187145347604966*sin(\t r)});
\draw [shift={(6.5865582655826564,-4.4069376693769)}] plot[domain=0.7483780475235196:2.2367655641740045,variable=\t]({1.0*0.4141999604236925*cos(\t r)+-0.0*0.4141999604236925*sin(\t r)},{0.0*0.4141999604236925*cos(\t r)+1.0*0.4141999604236925*sin(\t r)});
\draw [shift={(5.199024390243904,-3.691490514905354)}] plot[domain=3.5124839424024548:5.608444364956036,variable=\t]({1.0*0.4187145347604982*cos(\t r)+-0.0*0.4187145347604982*sin(\t r)},{0.0*0.4187145347604982*cos(\t r)+1.0*0.4187145347604982*sin(\t r)});
\draw [shift={(5.1339837398374,-4.320216802168226)}] plot[domain=0.7483780475235196:2.2367655641740063,variable=\t]({1.0*0.4141999604236925*cos(\t r)+-0.0*0.4141999604236925*sin(\t r)},{0.0*0.4141999604236925*cos(\t r)+1.0*0.4141999604236925*sin(\t r)});
\end{tikzpicture}

\caption{Support of a star relation}

%%%%%%%%%%%%%%%%%%%%%%%%%%%%%%%%%%%%%%%%%%%%%%%%%%%%%%%%%%%%%
%
% figure 5 : supports of two Lantern relations
%
%%%%%%%%%%%%%%%%%%%%%%%%%%%%%%%%%%%%%%%%%%%%%%%%%%%%%%%%%%%%%%

\definecolor{ffzzqq}{rgb}{1.0,0.6,0.0}
\definecolor{ffxfqq}{rgb}{1.0,0.4980392156862745,0.0}
\definecolor{ffqqqq}{rgb}{1.0,0.0,0.0}
\definecolor{qqqqff}{rgb}{0.0,0.0,1.0}
\begin{tikzpicture}[line cap=round,line join=round,>=triangle 45,x=1.0cm,y=1.0cm,scale=0.5]
\clip(-13.415637860082306,-3.0864197530864192) rectangle (9.08093278463649,6.74074074074074);
\draw [rotate around={0.8962503768615124:(3.0323343631436313,1.411504607046175)}] (3.0323343631436313,1.411504607046175) ellipse (4.796658876222521cm and 3.487953281160879cm);
\draw [shift={(0.8354146341463418,1.3975609756097556)}] plot[domain=0.0252471595441613:3.1668398131339544,variable=\t]({1.0*0.6956363282989587*cos(\t r)+-0.0*0.6956363282989587*sin(\t r)},{0.0*0.6956363282989587*cos(\t r)+1.0*0.6956363282989587*sin(\t r)});
\draw [shift={(0.8213658536585364,1.6536585365853655)}] plot[domain=-3.106884375364518:0.034708278225275,variable=\t]({1.0*0.7590913256585574*cos(\t r)+-0.0*0.7590913256585574*sin(\t r)},{0.0*0.7590913256585574*cos(\t r)+1.0*0.7590913256585574*sin(\t r)});
\draw(3.98,2.98) circle (0.7cm);
\draw(3.92,0.18) circle (0.7cm);
\draw [shift={(3.28,1.04)},color=qqqqff]  plot[domain=1.5816654639721657:2.7509856098921075,variable=\t]({1.0*1.9946977695846473*cos(\t r)+-0.0*1.9946977695846473*sin(\t r)},{0.0*1.9946977695846473*cos(\t r)+1.0*1.9946977695846473*sin(\t r)});
\draw [shift={(1.4,3.8)},dotted,color=qqqqff]  plot[domain=4.754916515004474:5.85017944555596,variable=\t]({1.0*2.0397846697173474*cos(\t r)+-0.0*2.0397846697173474*sin(\t r)},{0.0*2.0397846697173474*cos(\t r)+1.0*2.0397846697173474*sin(\t r)});
\draw [shift={(-0.78,0.1)},color=qqqqff]  plot[domain=0.879853098792494:2.279422598922568,variable=\t]({1.0*1.4628926098937511*cos(\t r)+-0.0*1.4628926098937511*sin(\t r)},{0.0*1.4628926098937511*cos(\t r)+1.0*1.4628926098937511*sin(\t r)});
\draw [shift={(-0.7577960975609748,2.485391869918491)},dotted,color=qqqqff]  plot[domain=4.098369943376749:5.359221458088224,variable=\t]({1.0*1.6180473198521663*cos(\t r)+-0.0*1.6180473198521663*sin(\t r)},{0.0*1.6180473198521663*cos(\t r)+1.0*1.6180473198521663*sin(\t r)});
\draw [shift={(2.6,1.62)},color=qqqqff]  plot[domain=-0.5286057039102916:0.46642152699005035,variable=\t]({1.0*1.4276371146451885*cos(\t r)+-0.0*1.4276371146451885*sin(\t r)},{0.0*1.4276371146451885*cos(\t r)+1.0*1.4276371146451885*sin(\t r)});
\draw [shift={(2.68,-1.24)},color=qqqqff]  plot[domain=-0.5286057039102925:0.46642152699005773,variable=\t]({1.0*1.427637114645185*cos(\t r)+-0.0*1.427637114645185*sin(\t r)},{0.0*1.427637114645185*cos(\t r)+1.0*1.427637114645185*sin(\t r)});
\draw [shift={(5.0,-1.26)},dotted,color=qqqqff]  plot[domain=2.5959375990628013:3.705853276176649,variable=\t]({1.0*1.310267148332736*cos(\t r)+-0.0*1.310267148332736*sin(\t r)},{0.0*1.310267148332736*cos(\t r)+1.0*1.310267148332736*sin(\t r)});
\draw [shift={(4.92,1.6)},dotted,color=qqqqff]  plot[domain=2.5959375990628017:3.7058532761766494,variable=\t]({1.0*1.310267148332736*cos(\t r)+-0.0*1.310267148332736*sin(\t r)},{0.0*1.310267148332736*cos(\t r)+1.0*1.310267148332736*sin(\t r)});
\draw [color=qqqqff](4.334705075445816,-0.5925925925925924) node[anchor=north west] {$\mathbf{\delta_2}$};
\draw [color=qqqqff](4.224965706447188,2.172839506172839) node[anchor=north west] {$\mathbf{\delta_3}$};
\draw [color=qqqqff](-1.207133058984911,2.543209876543209) node[anchor=north west] {$\mathbf{\alpha_1}$};
\draw [color=qqqqff](1.59122085048011,3.7283950617283943) node[anchor=north west] {$\mathbf{\alpha_3}$};
\draw [shift={(4.56,1.6)},dotted,color=ffqqqq]  plot[domain=2.4831672798064237:3.821006458997642,variable=\t]({1.0*1.3401492454200759*cos(\t r)+-0.0*1.3401492454200759*sin(\t r)},{0.0*1.3401492454200759*cos(\t r)+1.0*1.3401492454200759*sin(\t r)});
\draw [shift={(-1.64,5.58)},color=ffxfqq]  plot[domain=4.78196291545526:5.754057673255315,variable=\t]({1.0*5.753920402647225*cos(\t r)+-0.0*5.753920402647225*sin(\t r)},{0.0*5.753920402647225*cos(\t r)+1.0*5.753920402647225*sin(\t r)});
\draw [shift={(-2.92,7.7)},dotted,color=ffxfqq]  plot[domain=4.91370608875933:5.609224091893532,variable=\t]({1.0*8.001599840031993*cos(\t r)+-0.0*8.001599840031993*sin(\t r)},{0.0*8.001599840031993*cos(\t r)+1.0*8.001599840031993*sin(\t r)});
\draw [shift={(-1.24,2.48)},dotted,color=ffqqqq]  plot[domain=4.585796853104718:5.556542966497861,variable=\t]({1.0*2.2177466041006584*cos(\t r)+-0.0*2.2177466041006584*sin(\t r)},{0.0*2.2177466041006584*cos(\t r)+1.0*2.2177466041006584*sin(\t r)});
\draw [shift={(1.5,1.1)},color=ffqqqq]  plot[domain=3.230642236224291:5.497787143782138,variable=\t]({1.0*1.124455423749648*cos(\t r)+-0.0*1.124455423749648*sin(\t r)},{0.0*1.124455423749648*cos(\t r)+1.0*1.124455423749648*sin(\t r)});
\draw [color=ffqqqq](1.9478737997256517,-0.24691358024691354) node[anchor=north west] {$\mathbf{\sigma_2}$};
\draw [color=ffxfqq](1.4540466392318245,1.1358024691358022) node[anchor=north west] {$\mathbf{\gamma_2}$};
\draw [shift={(1.2620027434842251,3.4567901234567895)},color=ffqqqq]  plot[domain=3.9998274206111772:4.307034873281875,variable=\t]({1.0*4.224704279407645*cos(\t r)+-0.0*4.224704279407645*sin(\t r)},{0.0*4.224704279407645*cos(\t r)+1.0*4.224704279407645*sin(\t r)});
\draw [rotate around={0.8962503768615159:(-7.706815156746632,1.4781712737128385)}] (-7.706815156746632,1.4781712737128385) ellipse (4.796658876222525cm and 3.487953281160881cm);
\draw(-6.9191495198902615,2.7866666666666635) circle (0.7cm);
\draw(-6.9191495198902615,0.22666666666666324) circle (0.7cm);
\draw [shift={(-8.11914951989026,4.226666666666665)},color=qqqqff]  plot[domain=-0.5286057039102916:0.4664215269900558,variable=\t]({1.0*1.427637114645189*cos(\t r)+-0.0*1.427637114645189*sin(\t r)},{0.0*1.427637114645189*cos(\t r)+1.0*1.427637114645189*sin(\t r)});
\draw [shift={(-5.759149519890261,4.166666666666664)},dotted,color=qqqqff]  plot[domain=2.5860989183242493:3.6536938574755973,variable=\t]({1.0*1.3652838532700813*cos(\t r)+-0.0*1.3652838532700813*sin(\t r)},{0.0*1.3652838532700813*cos(\t r)+1.0*1.3652838532700813*sin(\t r)});
\draw [shift={(-8.219149519890262,-1.1933333333333367)},color=qqqqff]  plot[domain=-0.5286057039102916:0.46642152699005734,variable=\t]({1.0*1.4276371146451872*cos(\t r)+-0.0*1.4276371146451872*sin(\t r)},{0.0*1.4276371146451872*cos(\t r)+1.0*1.4276371146451872*sin(\t r)});
\draw [shift={(-5.839149519890261,-1.2533333333333367)},dotted,color=qqqqff]  plot[domain=2.5860989183242493:3.653693857475597,variable=\t]({1.0*1.3652838532700813*cos(\t r)+-0.0*1.3652838532700813*sin(\t r)},{0.0*1.3652838532700813*cos(\t r)+1.0*1.3652838532700813*sin(\t r)});
\draw [shift={(-9.88373488574392,1.3642276422764188)}] plot[domain=0.025247159544161284:3.1668398131339544,variable=\t]({1.0*0.6956363282989592*cos(\t r)+-0.0*0.6956363282989592*sin(\t r)},{0.0*0.6956363282989592*cos(\t r)+1.0*0.6956363282989592*sin(\t r)});
\draw [shift={(-9.897783666231724,1.620325203252029)}] plot[domain=-3.106884375364518:0.034708278225275,variable=\t]({1.0*0.7590913256585577*cos(\t r)+-0.0*0.7590913256585577*sin(\t r)},{0.0*0.7590913256585577*cos(\t r)+1.0*0.7590913256585577*sin(\t r)});
\draw [shift={(-7.519149519890261,1.1266666666666634)},color=qqqqff]  plot[domain=1.5816654639721666:2.7509856098921075,variable=\t]({1.0*1.9946977695846495*cos(\t r)+-0.0*1.9946977695846495*sin(\t r)},{0.0*1.9946977695846495*cos(\t r)+1.0*1.9946977695846495*sin(\t r)});
\draw [shift={(-9.419149519890262,3.9066666666666636)},dotted,color=qqqqff]  plot[domain=4.754916515004474:5.850179445555959,variable=\t]({1.0*2.0397846697173483*cos(\t r)+-0.0*2.0397846697173483*sin(\t r)},{0.0*2.0397846697173483*cos(\t r)+1.0*2.0397846697173483*sin(\t r)});
\draw [shift={(-7.739149519890262,1.6666666666666634)},color=qqqqff]  plot[domain=3.5418333106237676:4.826829285654754,variable=\t]({1.0*1.8361748570397518*cos(\t r)+-0.0*1.8361748570397518*sin(\t r)},{0.0*1.8361748570397518*cos(\t r)+1.0*1.8361748570397518*sin(\t r)});
\draw [shift={(-9.399149519890262,-1.3333333333333368)},dotted,color=qqqqff]  plot[domain=0.5719181998612628:1.593014892121615,variable=\t]({1.0*2.2431429792381365*cos(\t r)+-0.0*2.2431429792381365*sin(\t r)},{0.0*2.2431429792381365*cos(\t r)+1.0*2.2431429792381365*sin(\t r)});
\draw [shift={(-8.979149519890262,1.3466666666666631)},color=ffzzqq]  plot[domain=2.1211822744716877:4.2522256778410235,variable=\t]({1.0*2.06494551986245*cos(\t r)+-0.0*2.06494551986245*sin(\t r)},{0.0*2.06494551986245*cos(\t r)+1.0*2.06494551986245*sin(\t r)});
\draw [shift={(-8.319149519890262,0.6266666666666634)},color=ffxfqq]  plot[domain=1.1988786955441153:2.2397030294976954,variable=\t]({1.0*3.02694565527695*cos(\t r)+-0.0*3.02694565527695*sin(\t r)},{0.0*3.02694565527695*cos(\t r)+1.0*3.02694565527695*sin(\t r)});
\draw [shift={(-8.819149519890262,1.7466666666666635)},color=ffxfqq]  plot[domain=4.245144264596339:5.307711154414342,variable=\t]({1.0*2.4865236777477104*cos(\t r)+-0.0*2.4865236777477104*sin(\t r)},{0.0*2.4865236777477104*cos(\t r)+1.0*2.4865236777477104*sin(\t r)});
\draw [shift={(-8.50480109739369,0.6913580246913579)},dotted,color=ffxfqq]  plot[domain=1.1391893657896666:2.324726004493835,variable=\t]({1.0*3.0555156816003453*cos(\t r)+-0.0*3.0555156816003453*sin(\t r)},{0.0*3.0555156816003453*cos(\t r)+1.0*3.0555156816003453*sin(\t r)});
\draw [shift={(-8.943758573388203,1.3827160493827158)},dotted,color=ffxfqq]  plot[domain=2.352816124666852:4.024686485661118,variable=\t]({1.0*2.296897686309359*cos(\t r)+-0.0*2.296897686309359*sin(\t r)},{0.0*2.296897686309359*cos(\t r)+1.0*2.296897686309359*sin(\t r)});
\draw [shift={(-8.916323731138547,1.5555555555555554)},dotted,color=ffxfqq]  plot[domain=3.987426682923491:5.3440796147236584,variable=\t]({1.0*2.441016534521877*cos(\t r)+-0.0*2.441016534521877*sin(\t r)},{0.0*2.441016534521877*cos(\t r)+1.0*2.441016534521877*sin(\t r)});
\draw [shift={(-7.659149519890262,1.326666666666663)},dotted,color=ffqqqq]  plot[domain=1.4885083370202148:2.7973739706073255,variable=\t]({1.0*1.9465867563507164*cos(\t r)+-0.0*1.9465867563507164*sin(\t r)},{0.0*1.9465867563507164*cos(\t r)+1.0*1.9465867563507164*sin(\t r)});
\draw [shift={(-6.099149519890261,-1.2133333333333367)},dotted,color=ffqqqq]  plot[domain=2.553590050042226:3.6820121538603767,variable=\t]({1.0*1.3701094846763164*cos(\t r)+-0.0*1.3701094846763164*sin(\t r)},{0.0*1.3701094846763164*cos(\t r)+1.0*1.3701094846763164*sin(\t r)});
\draw [shift={(-9.11914951989026,1.6266666666666634)},color=ffqqqq]  plot[domain=3.982036864574819:5.460304707090476,variable=\t]({1.0*2.7802405398586254*cos(\t r)+-0.0*2.7802405398586254*sin(\t r)},{0.0*2.7802405398586254*cos(\t r)+1.0*2.7802405398586254*sin(\t r)});
\draw [shift={(-6.976844993141292,6.121975308641973)},color=ffqqqq]  plot[domain=4.223986754343837:4.686285250494962,variable=\t]({1.0*8.050912663979569*cos(\t r)+-0.0*8.050912663979569*sin(\t r)},{0.0*8.050912663979569*cos(\t r)+1.0*8.050912663979569*sin(\t r)});
\draw [color=qqqqff](-8.88888888888889,2.9123456790123453) node[anchor=north west] {$\mathbf{\alpha_3}$};
\draw [color=qqqqff](-8.906584362139919,0.9135802469135801) node[anchor=north west] {$\mathbf{\alpha_2}$};
\draw [color=qqqqff](-6.584362139917696,4.6913580246913575) node[anchor=north west] {$\mathbf{\delta_1}$};
\draw [color=qqqqff](-6.748971193415638,-0.6666666666666665) node[anchor=north west] {$\mathbf{\delta_2}$};
\draw [color=ffqqqq](-8.23045267489712,-0.8641975308641974) node[anchor=north west] {$\mathbf{\sigma_1}$};
\draw [color=ffxfqq](-10.617283950617285,0.8888888888888887) node[anchor=north west] {$\mathbf{\gamma_1}$};
\draw [shift={(-9.492455418381345,1.259259259259259)},color=ffqqqq]  plot[domain=0.7713722861009187:1.7662871270435776,variable=\t]({1.0*2.869110215131463*cos(\t r)+-0.0*2.869110215131463*sin(\t r)},{0.0*2.869110215131463*cos(\t r)+1.0*2.869110215131463*sin(\t r)});
\draw [shift={(-9.218106995884774,1.6543209876543206)},color=ffqqqq]  plot[domain=1.9018920253307636:2.8806483113702104,variable=\t]({1.0*2.558351419826457*cos(\t r)+-0.0*2.558351419826457*sin(\t r)},{0.0*2.558351419826457*cos(\t r)+1.0*2.558351419826457*sin(\t r)});
\draw [shift={(-9.026063100137176,1.3086419753086418)},color=ffqqqq]  plot[domain=2.7814807524256926:4.063355219803807,variable=\t]({1.0*2.872889124387203*cos(\t r)+-0.0*2.872889124387203*sin(\t r)},{0.0*2.872889124387203*cos(\t r)+1.0*2.872889124387203*sin(\t r)});
\draw [shift={(-10.480109739369,1.5802469135802466)},color=ffqqqq]  plot[domain=0.3733262082953229:2.816604346383827,variable=\t]({1.0*1.0320692441921246*cos(\t r)+-0.0*1.0320692441921246*sin(\t r)},{0.0*1.0320692441921246*cos(\t r)+1.0*1.0320692441921246*sin(\t r)});
\draw [shift={(0.7133058984910837,1.7037037037037033)},color=ffqqqq]  plot[domain=4.229152073168019:5.497162534244378,variable=\t]({1.0*2.4045967647383253*cos(\t r)+-0.0*2.4045967647383253*sin(\t r)},{0.0*2.4045967647383253*cos(\t r)+1.0*2.4045967647383253*sin(\t r)});
\draw [shift={(4.471879286694102,-1.8518518518518514)},color=ffqqqq]  plot[domain=1.9164630816653598:2.408777551803287,variable=\t]({1.0*2.77789931501165*cos(\t r)+-0.0*2.77789931501165*sin(\t r)},{0.0*2.77789931501165*cos(\t r)+1.0*2.77789931501165*sin(\t r)});
\draw [shift={(5.816186556927298,0.3703703703703703)},color=ffqqqq]  plot[domain=2.4049662848393756:2.911386657641128,variable=\t]({1.0*3.13923790849293*cos(\t r)+-0.0*3.13923790849293*sin(\t r)},{0.0*3.13923790849293*cos(\t r)+1.0*3.13923790849293*sin(\t r)});
\draw [shift={(0.5761316872427984,1.8518518518518514)},color=ffqqqq]  plot[domain=5.550408332939107:5.946141315458713,variable=\t]({1.0*2.312569517196533*cos(\t r)+-0.0*2.312569517196533*sin(\t r)},{0.0*2.312569517196533*cos(\t r)+1.0*2.312569517196533*sin(\t r)});
\draw [shift={(-8.751714677640605,1.2345679012345676)},color=ffqqqq]  plot[domain=2.8970987031443225:3.7973662221538524,variable=\t]({1.0*2.7893968975903918*cos(\t r)+-0.0*2.7893968975903918*sin(\t r)},{0.0*2.7893968975903918*cos(\t r)+1.0*2.7893968975903918*sin(\t r)});
\draw (-5.148010973936901,2.518518518518518) node[anchor=north west] {$\mathbf{\ldots}$};
\draw (-5.4900548696845,1.185185185185185) node[anchor=north west] {$f-3$};
\draw (-5.292880658436215,1.776172839506168) node[anchor=north west] {$ \underbrace{}$};
\draw [shift={(-3.9206397002241653,2.02710027100271)}] plot[domain=3.512483942402456:5.6084443649560365,variable=\t]({1.0*0.41871453476049647*cos(\t r)+-0.0*0.41871453476049647*sin(\t r)},{0.0*0.41871453476049647*cos(\t r)+1.0*0.41871453476049647*sin(\t r)});
\draw [shift={(-3.985680350630668,1.3983739837398363)}] plot[domain=0.7483780475235196:2.2367655641740045,variable=\t]({1.0*0.4141999604236925*cos(\t r)+-0.0*0.4141999604236925*sin(\t r)},{0.0*0.4141999604236925*cos(\t r)+1.0*0.4141999604236925*sin(\t r)});
\draw [shift={(-5.37321422596942,2.1138211382113825)}] plot[domain=3.512483942402456:5.608444364956035,variable=\t]({1.0*0.4187145347604984*cos(\t r)+-0.0*0.4187145347604984*sin(\t r)},{0.0*0.4187145347604984*cos(\t r)+1.0*0.4187145347604984*sin(\t r)});
\draw [shift={(-5.438254876375924,1.4850948509485098)}] plot[domain=0.7483780475235196:2.2367655641740063,variable=\t]({1.0*0.4141999604236925*cos(\t r)+-0.0*0.4141999604236925*sin(\t r)},{0.0*0.4141999604236925*cos(\t r)+1.0*0.4141999604236925*sin(\t r)});
\draw (5.731056241426612,2.518518518518518) node[anchor=north west] {$\mathbf{\ldots}$};
\draw (5.429012345679013,1.185185185185185) node[anchor=north west] {$f-3$};
\draw (5.616186556927298,1.8106172839506168) node[anchor=north west] {$ \underbrace{}$};
\draw [shift={(6.998427515139347,2.02710027100271)}] plot[domain=3.5124839424024565:5.608444364956036,variable=\t]({1.0*0.4187145347604966*cos(\t r)+-0.0*0.4187145347604966*sin(\t r)},{0.0*0.4187145347604966*cos(\t r)+1.0*0.4187145347604966*sin(\t r)});
\draw [shift={(6.933386864732844,1.3983739837398361)}] plot[domain=0.7483780475235196:2.2367655641740045,variable=\t]({1.0*0.4141999604236925*cos(\t r)+-0.0*0.4141999604236925*sin(\t r)},{0.0*0.4141999604236925*cos(\t r)+1.0*0.4141999604236925*sin(\t r)});
\draw [shift={(5.545852989394092,2.113821138211382)}] plot[domain=3.5124839424024548:5.608444364956036,variable=\t]({1.0*0.4187145347604982*cos(\t r)+-0.0*0.4187145347604982*sin(\t r)},{0.0*0.4187145347604982*cos(\t r)+1.0*0.4187145347604982*sin(\t r)});
\draw [shift={(5.480812338987588,1.4850948509485096)}] plot[domain=0.7483780475235196:2.2367655641740063,variable=\t]({1.0*0.4141999604236925*cos(\t r)+-0.0*0.4141999604236925*sin(\t r)},{0.0*0.4141999604236925*cos(\t r)+1.0*0.4141999604236925*sin(\t r)});
\end{tikzpicture}
\caption{Supports of two lantern relations}
\end{figure}

%%%%%%%%%%%%%%%%%%%%%%%%%%%%%%%%%%%%%%%%%%%%%%%
%
% prop : \alpha_1^4 \alpha_2^2 [,][,][,] = 1
%
%%%%%%%%%%%%%%%%%%%%%%%%%%%%%%%%%%%%%%%%%%%%%%%%%%

\begin{proposition}
Let $f\geq3$ and let $\{\alpha_1, \alpha_2\}$ be any pair of nonseparating simple closed curves on $\Sigma_f$ such that $\Sigma_f-\alpha_1-\alpha_2$ is connected. 
Then there exists a genus $f$ Lefschetz fibration $X$ over $\Sigma_3$ 
which has six singular fibers, four of which have monodromy $t_{\alpha_1}$ and two of which have monodromy $t_{\alpha_2}$.

\begin{proof}
We use the star relation $E:=t_{\delta_3}^{-1}t_{\delta_2}^{-1}t_{\delta_1}^{-1}(t_{\alpha_1}t_{\alpha_2}t_{\alpha_3}t_{\beta})^{3}$
supported on $\Sigma_1^{3}\hookrightarrow \Sigma_f$ (Figure $4$). Also, consider the following lantern relations whose supports are given in Figure $5$ :  
$L_1:= t_{\alpha_3}^{-1}t_{\alpha_2}^{-1}t_{\delta_1}^{-1}t_{\delta_2}^{-1}t_{\sigma_1}t_{\alpha_1}t_{\gamma_1}$,
$L_2:=t_{\alpha_3}^{-1}t_{\alpha_1}^{-1}t_{\delta_2}^{-1}t_{\delta_3}^{-1}$ $t_{\sigma_2}t_{\alpha_2}t_{\gamma_2}$. Let $W_0:=t_{\beta}(t_{\alpha_1}t_{\alpha_2}t_{\alpha_3}t_{\beta})^2$,
$W_1:=t_{\beta}t_{\alpha_1}t_{\alpha_2}t_{\alpha_3}t_{\beta}$, and
$W_2:=t_{\beta}$.
Then, by using commutativity relations and braid relations,
\begin{eqnarray*}
1&=&E\cdot(W_0^{-1}L_1W_0)\cdot(W_1^{-1}L_1W_1)\cdot(W_2^{-1}L_2W_2)\\
&=& t_{\delta_3}^{-1}t_{\delta_2}^{-1}t_{\delta_1}^{-1}t_{\alpha_1}
t_{\delta_1}^{-1}t_{\delta_2}^{-1}t_{\sigma_1}t_{\alpha_1}t_{\gamma_1}
t_{\beta}t_{\alpha_1}t_{\delta_1}^{-1}t_{\delta_2}^{-1}t_{\sigma_1}t_{\alpha_1}t_{\gamma_1}
t_{\beta}t_{\alpha_2}t_{\delta_2}^{-1}t_{\delta_3}^{-1}t_{\sigma_2}t_{\alpha_2}t_{\gamma_2}t_{\beta}\\ 
&=& t_{\alpha_1}t_{\delta_1}^{-1}t_{\sigma_1}t_{\delta_2}^{-1}t_{\alpha_1}t_{\gamma_1}t_{\delta_1}^{-1}t_{\beta}
t_{\alpha_1}t_{\delta_1}^{-1}t_{\sigma_1}t_{\delta_2}^{-1}t_{\alpha_1}t_{\gamma_1}t_{\delta_2}^{-1}t_{\beta}
t_{\alpha_2}t_{\delta_2}^{-1}t_{\sigma_2}t_{\delta_3}^{-1}t_{\alpha_2}t_{\gamma_2}t_{\delta_3}^{-1}t_{\beta}\\
&=&
t_{\alpha_1}^{2}\{t_{\delta_1}^{-1}t_{t_{\alpha_1}^{-1}(\sigma_1)}
t_{\delta_2}^{-1}t_{\gamma_1}t_{\delta_1}^{-1}t_{\beta}\}
t_{\alpha_1}^{2}\{t_{\delta_1}^{-1}
t_{t_{\alpha_1}^{-1}(\sigma_1)}
t_{\delta_2}^{-1}t_{\gamma_1}t_{\delta_2}^{-1}t_{\beta}\}\\
& & t_{\alpha_2}^{2}\{t_{\delta_2}^{-1}
t_{t_{\alpha_2}^{-1}(\sigma_2)}
t_{\delta_3}^{-1}t_{\gamma_2}t_{\delta_3}^{-1}t_{\beta}\}\\
&=& 
t_{\alpha_1}^{2}[t_{\delta_1}^{-1}t_{t_{\alpha_1}^{-1}(\sigma_1)}t_{\delta_2}^{-1},\phi_1]
t_{\alpha_1}^{2}[t_{\delta_1}^{-1}t_{t_{\alpha_1}^{-1}(\sigma_1)}t_{\delta_2}^{-1},\phi_2]
t_{\alpha_2}^{2}[t_{\delta_2}^{-1}t_{t_{\alpha_2}^{-1}(\sigma_2)}t_{\delta_3}^{-1},\phi_3]
\end{eqnarray*}

For the last equality, we need to verify that there exists a self-homeomorphism
$\phi_1$ of $\Sigma_f$ sending $\delta_1$, $t_{\alpha_1}^{-1}(\sigma_1)$, and
$\delta_2$ to $\beta$, $\delta_1$, and $\gamma_1$, respectively.
First, it is easy to check that $\sigma_1=t_{\beta}^{-1}t_{\alpha_2}^{-1}t_{\alpha_1}t_{\alpha_3}^{-1}(\beta)$.
Hence, the self-homeomorphism $t_{\alpha_3}t_{\alpha_1}^{-1}t_{\alpha_2}t_{\beta}t_{\alpha_1}$ 
sends  $\delta_1$, $t_{\alpha_1}^{-1}(\sigma_1)$, and
$\delta_2$ to $\delta_1$, $\beta$, and $\delta_2$, respectively. Also, there exists a homeomorphism sending $\delta_1$, $\beta$, and $\delta_2$ to $\beta$, $\delta_1$, and $\gamma_1$, respectively,
because both $\Sigma_f-\delta_1-\beta-\delta_2$ and 
$\Sigma_f-\beta-\delta_1-\gamma_1$ are homeomorphic to $\Sigma_{f-3}^{6}$.
The composition of these two homeomorphisms is the required $\phi_1$.
The existence of $\phi_2$ and $\phi_3$ can be proven in a similar way because $\sigma_2=t_{\beta}^{-1}t_{\alpha_1}^{-1}t_{\alpha_3}^{-1}t_{\alpha_2}(\beta)$. Finally, we get the required Lefschetz fibration over $\Sigma_3$ with fiber $\Sigma_f$ whose monodromy factorization is given by $[t_{\delta_1}^{-1}t_{t_{\alpha_1}^{-1}(\sigma_1)}t_{\delta_2}^{-1},\phi_1]^{t_{\alpha_1}^{-2}}[t_{\delta_1}^{-1}t_{t_{\alpha_1}^{-1}(\sigma_1)}t_{\delta_2}^{-1},\phi_2]^{t_{\alpha_1}^{-4}}[t_{\delta_2}^{-1}t_{t_{\alpha_2}^{-1}(\sigma_2)}t_{\delta_3}^{-1},\phi_3]^{t_{\alpha_2}^{-2}t_{\alpha_1}^{-4}}t_{\alpha_1}^{4}t_{\alpha_2}^{2}=1$.

\end{proof}
\end{proposition}

%%%%%%%%%%%%%%%%%%%%%%%%%%%%%%%%%%%%%%%%%%%%%
%
% figure 6 : torus with four holed relation
%
%%%%%%%%%%%%%%%%%%%%%%%%%%%%%%%%%%%%%%%%%%%%%%

\definecolor{qqqqff}{rgb}{0.0,0.0,1.0}
\begin{figure}
\begin{tikzpicture}[line cap=round,line join=round,>=triangle 45,x=1.0cm,y=1.0cm,scale=0.6]
\clip(-4.320000000000002,-2.9597399999999965) rectangle (9.972000000000003,4.169879999999999);
\draw (0.41279284155557255,3.7017582073377846)-- (5.352792841555573,3.7017582073377846);
\draw (0.45455511074884675,0.5617582073377856)-- (5.394555110748847,0.5617582073377856);
\draw [shift={(0.4427928415555726,2.1417582073377854)}] plot[domain=-1.5490606199531038:1.5925320336366893,variable=\t]({1.0*0.46010868281309364*cos(\t r)+-0.0*0.46010868281309364*sin(\t r)},{0.0*0.46010868281309364*cos(\t r)+1.0*0.46010868281309364*sin(\t r)});
\draw [shift={(5.3842912663163505,2.143251523184772)}] plot[domain=1.5640054032352477:4.705598056825041,variable=\t]({1.0*0.46010868281309547*cos(\t r)+-0.0*0.46010868281309547*sin(\t r)},{0.0*0.46010868281309547*cos(\t r)+1.0*0.46010868281309547*sin(\t r)});
\draw [shift={(1.3127928415555725,3.161758207337786)},color=qqqqff]  plot[domain=2.60117315331921:3.7083218711133,variable=\t]({1.0*1.0495713410721532*cos(\t r)+-0.0*1.0495713410721532*sin(\t r)},{0.0*1.0495713410721532*cos(\t r)+1.0*1.0495713410721532*sin(\t r)});
\draw [shift={(-0.44829588599976106,3.138708575293169)},dash pattern=on 2pt off 2pt,color=qqqqff]  plot[domain=-0.561216156696057:0.5459325610980335,variable=\t]({1.0*1.0495713410721543*cos(\t r)+-0.0*1.0495713410721543*sin(\t r)},{0.0*1.0495713410721543*cos(\t r)+1.0*1.0495713410721543*sin(\t r)});
\draw [shift={(6.253673976152209,3.142639341934423)},dash pattern=on 2pt off 2pt,color=qqqqff]  plot[domain=2.6011731533192086:3.7083218711133,variable=\t]({1.0*1.0495713410721528*cos(\t r)+-0.0*1.0495713410721528*sin(\t r)},{0.0*1.0495713410721528*cos(\t r)+1.0*1.0495713410721528*sin(\t r)});
\draw [shift={(4.472792841555573,3.163520476531059)},color=qqqqff]  plot[domain=-0.561216156696057:0.5459325610980331,variable=\t]({1.0*1.049571341072155*cos(\t r)+-0.0*1.049571341072155*sin(\t r)},{0.0*1.049571341072155*cos(\t r)+1.0*1.049571341072155*sin(\t r)});
\draw [shift={(4.500906922895644,1.1256783276339943)},color=qqqqff]  plot[domain=-0.561216156696057:0.5459325610980341,variable=\t]({1.0*1.049571341072155*cos(\t r)+-0.0*1.049571341072155*sin(\t r)},{0.0*1.049571341072155*cos(\t r)+1.0*1.049571341072155*sin(\t r)});
\draw [shift={(6.2752423550794605,1.1130760584407202)},dash pattern=on 2pt off 2pt,color=qqqqff]  plot[domain=2.601173153319209:3.708321871113299,variable=\t]({1.0*1.0495713410721543*cos(\t r)+-0.0*1.0495713410721543*sin(\t r)},{0.0*1.0495713410721543*cos(\t r)+1.0*1.0495713410721543*sin(\t r)});
\draw [shift={(1.3292945902730817,1.1256783276339943)},color=qqqqff]  plot[domain=2.601173153319209:3.7083218711132995,variable=\t]({1.0*1.0495713410721543*cos(\t r)+-0.0*1.0495713410721543*sin(\t r)},{0.0*1.0495713410721543*cos(\t r)+1.0*1.0495713410721543*sin(\t r)});
\draw [shift={(-0.4456018504410802,1.1256783276339943)},dash pattern=on 2pt off 2pt,color=qqqqff]  plot[domain=-0.561216156696057:0.5459325610980341,variable=\t]({1.0*1.0495713410721546*cos(\t r)+-0.0*1.0495713410721546*sin(\t r)},{0.0*1.0495713410721546*cos(\t r)+1.0*1.0495713410721546*sin(\t r)});
\draw [shift={(2.8545790141360903,1.3127676749329167)}] plot[domain=0.9943084968661196:2.124927869156937,variable=\t]({1.0*1.022470081022615*cos(\t r)+-0.0*1.022470081022615*sin(\t r)},{0.0*1.022470081022615*cos(\t r)+1.0*1.022470081022615*sin(\t r)});
\draw [shift={(2.877780274799021,2.6094861333840877)}] plot[domain=3.5264019301356546:5.873776451165018,variable=\t]({1.0*0.7008914603952072*cos(\t r)+-0.0*0.7008914603952072*sin(\t r)},{0.0*0.7008914603952072*cos(\t r)+1.0*0.7008914603952072*sin(\t r)});
\draw [rotate around={0.0:(2.866179644467555,2.177246647233697)},color=qqqqff] (2.866179644467555,2.177246647233697) ellipse (1.0592596337921492cm and 0.7882124050994147cm);
\draw [shift={(1.624912199000789,3.3987930211369735)},color=qqqqff]  plot[domain=4.192395457678181:5.240207521625644,variable=\t]({1.0*1.494221008000354*cos(\t r)+-0.0*1.494221008000354*sin(\t r)},{0.0*1.494221008000354*cos(\t r)+1.0*1.494221008000354*sin(\t r)});
\draw [shift={(1.624912199000789,-0.1718809948879889)},dash pattern=on 2pt off 2pt,color=qqqqff]  plot[domain=1.2532429248175205:1.8750988121290637,variable=\t]({1.0*2.3860469745561397*cos(\t r)+-0.0*2.3860469745561397*sin(\t r)},{0.0*2.3860469745561397*cos(\t r)+1.0*2.3860469745561397*sin(\t r)});
\draw [shift={(4.153849611260182,3.417586042273947)},color=qqqqff]  plot[domain=4.192395457678181:5.240207521625644,variable=\t]({1.0*1.4942210080003535*cos(\t r)+-0.0*1.4942210080003535*sin(\t r)},{0.0*1.4942210080003535*cos(\t r)+1.0*1.4942210080003535*sin(\t r)});
\draw [shift={(4.1306483505972516,-0.13429495261404178)},dash pattern=on 2pt off 2pt,color=qqqqff]  plot[domain=1.2532429248175196:1.8750988121290637,variable=\t]({1.0*2.3860469745561366*cos(\t r)+-0.0*2.3860469745561366*sin(\t r)},{0.0*2.3860469745561366*cos(\t r)+1.0*2.3860469745561366*sin(\t r)});
\draw [shift={(1.7177172416525102,1.2563886115219964)},color=qqqqff]  plot[domain=-0.507348192277286:0.4826582821063363,variable=\t]({1.0*1.3537935596819786*cos(\t r)+-0.0*1.3537935596819786*sin(\t r)},{0.0*1.3537935596819786*cos(\t r)+1.0*1.3537935596819786*sin(\t r)});
\draw [shift={(3.991440786619672,1.2000095481110757)},dash pattern=on 2pt off 2pt,color=qqqqff]  plot[domain=2.5937824158288696:3.672065024925609,variable=\t]({1.0*1.2946773687218243*cos(\t r)+-0.0*1.2946773687218243*sin(\t r)},{0.0*1.2946773687218243*cos(\t r)+1.0*1.2946773687218243*sin(\t r)});
\draw [shift={(1.6945159809895796,3.0417256195344775)},color=qqqqff]  plot[domain=-0.507348192277286:0.4826582821063372,variable=\t]({1.0*1.3537935596819777*cos(\t r)+-0.0*1.3537935596819777*sin(\t r)},{0.0*1.3537935596819777*cos(\t r)+1.0*1.3537935596819777*sin(\t r)});
\draw [shift={(3.9682395259567405,3.022932598397504)},dash pattern=on 2pt off 2pt,color=qqqqff]  plot[domain=2.59378241582887:3.6988078867721947,variable=\t]({1.0*1.2946773687218238*cos(\t r)+-0.0*1.2946773687218238*sin(\t r)},{0.0*1.2946773687218238*cos(\t r)+1.0*1.2946773687218238*sin(\t r)});
\draw [color=qqqqff](-0.64200000000000114,3.499199999999999) node[anchor=north west] {$\mathbf{\delta_1}$};
\draw [color=qqqqff](-0.6060000000000011,1.5017400000000007) node[anchor=north west] {$\mathbf{\delta_2}$};
\draw [color=qqqqff](5.670000000000001,1.4871600000000007) node[anchor=north west] {$\mathbf{\delta_3}$};
\draw [color=qqqqff](5.598000000000001,3.499199999999999) node[anchor=north west] {$\mathbf{\delta_4}$};
\draw [color=qqqqff](1.0259999999999994,2.8576799999999998) node[anchor=north west] {$\mathbf{\alpha_1}$};
\draw [color=qqqqff](1.7259999999999995,1.4434200000000008) node[anchor=north west] {$\mathbf{\alpha_2}$};
\draw [color=qqqqff](4.104,1.9245600000000005) node[anchor=north west] {$\mathbf{\alpha_3}$};
\draw [color=qqqqff](3.096,3.644999999999999) node[anchor=north west] {$\mathbf{\alpha_4}$};
\draw [color=qqqqff](3.7800000000000002,3.1638599999999997) node[anchor=north west] {$\mathbf{\beta}$};
\draw [shift={(0.45591738043401725,2.1418562808577675)}] plot[domain=1.5640054032352477:4.705598056825041,variable=\t]({1.0*0.46010868281309547*cos(\t r)+-0.0*0.46010868281309547*sin(\t r)},{0.0*0.46010868281309547*cos(\t r)+1.0*0.46010868281309547*sin(\t r)});
\draw [shift={(5.397415805194795,2.143349596704754)}] plot[domain=-1.549060619953103:1.5925320336366897,variable=\t]({1.0*0.4601086828130943*cos(\t r)+-0.0*0.4601086828130943*sin(\t r)},{0.0*0.4601086828130943*cos(\t r)+1.0*0.4601086828130943*sin(\t r)});
\draw [shift={(0.288,0.61236)}] plot[domain=1.535704542048115:4.687199046393073,variable=\t]({1.0*3.07827515085965*cos(\t r)+-0.0*3.07827515085965*sin(\t r)},{0.0*3.07827515085965*cos(\t r)+1.0*3.07827515085965*sin(\t r)});
\draw [shift={(5.652,0.62694)}] plot[domain=-1.5825791835642757:1.6557554254115578,variable=\t]({1.0*3.0782751508596524*cos(\t r)+-0.0*3.0782751508596524*sin(\t r)},{0.0*3.0782751508596524*cos(\t r)+1.0*3.0782751508596524*sin(\t r)});
\draw [shift={(0.4351245388784445,0.09559807351998309)}] plot[domain=1.5640054032352482:4.705598056825041,variable=\t]({1.0*0.4601086828130966*cos(\t r)+-0.0*0.4601086828130966*sin(\t r)},{0.0*0.4601086828130966*cos(\t r)+1.0*0.4601086828130966*sin(\t r)});
\draw [shift={(5.404555110748846,0.10175820733778485)}] plot[domain=-1.5490606199531065:1.5925320336366868,variable=\t]({1.0*0.4601086828130943*cos(\t r)+-0.0*0.4601086828130943*sin(\t r)},{0.0*0.4601086828130943*cos(\t r)+1.0*0.4601086828130943*sin(\t r)});
\draw (0.432,-0.3645)-- (5.382,-0.3645);
\draw (0.19798306996459497,-2.4645987024701848)-- (5.615729964076996,-2.4511214660682397);
\draw [shift={(1.782,-0.96228)}] plot[domain=3.5124839424024508:5.608444364956038,variable=\t]({1.0*0.4187145347604975*cos(\t r)+-0.0*0.4187145347604975*sin(\t r)},{0.0*0.4187145347604975*cos(\t r)+1.0*0.4187145347604975*sin(\t r)});
\draw [shift={(1.734959349593495,-1.5618462872628707)}] plot[domain=0.748378047523519:2.2367655641740054,variable=\t]({1.0*0.4141999604236928*cos(\t r)+-0.0*0.4141999604236928*sin(\t r)},{0.0*0.4141999604236928*cos(\t r)+1.0*0.4141999604236928*sin(\t r)});
\draw [shift={(3.618,-1.00602)}] plot[domain=3.5124839424024503:5.608444364956041,variable=\t]({1.0*0.41871453476049647*cos(\t r)+-0.0*0.41871453476049647*sin(\t r)},{0.0*0.41871453476049647*cos(\t r)+1.0*0.41871453476049647*sin(\t r)});
\draw [shift={(3.570959349593496,-1.6055862872628701)}] plot[domain=0.7483780475235203:2.2367655641740063,variable=\t]({1.0*0.41419996042369217*cos(\t r)+-0.0*0.41419996042369217*sin(\t r)},{0.0*0.41419996042369217*cos(\t r)+1.0*0.41419996042369217*sin(\t r)});
\draw (2.176,-0.9185399999999977) node[anchor=north west] {$\mathbf{\ldots}$};
\draw (1.928,-1.5204399999999973) node[anchor=north west] {$f-4$};
\draw (2.004,-1.1393799999999977) node[anchor=north west] {$\underbrace{}$};
\end{tikzpicture}

\caption{Support of a four-holed torus relation embedded in a genus $4$ surface}

%%%%%%%%%%%%%%%%%%%%%%%%%%%%%%%%%%%%%%%%%%%%%
%
% figure 7: two lantern relations
%
%%%%%%%%%%%%%%%%%%%%%%%%%%%%%%%%%%%%%%%%%%%%%

\definecolor{ffqqqq}{rgb}{1.0,0.0,0.0}
\definecolor{ffxfqq}{rgb}{1.0,0.4980392156862745,0.0}
\begin{tikzpicture}[line cap=round,line join=round,>=triangle 45,x=1.0cm,y=1.0cm,scale=0.75]
\clip(-1.777777777777777,-0.016199999999997234) rectangle (11.488888888888903,4.746600000000002);
\draw (-0.12145119256580517,0.658646363309082)-- (4.818548807434194,0.658646363309082);
\draw [shift={(3.924900619580992,1.2413595047422636)}] plot[domain=-0.5612161566960561:0.5459325610980339,variable=\t]({1.0*1.0495713410721554*cos(\t r)+-0.0*1.0495713410721554*sin(\t r)},{0.0*1.0495713410721554*cos(\t r)+1.0*1.0495713410721554*sin(\t r)});
\draw [shift={(5.704437312427739,1.218390256685963)},dash pattern=on 1pt off 1pt]  plot[domain=2.60117315331921:3.7083218711132995,variable=\t]({1.0*1.0495713410721559*cos(\t r)+-0.0*1.0495713410721559*sin(\t r)},{0.0*1.0495713410721559*cos(\t r)+1.0*1.0495713410721559*sin(\t r)});
\draw [shift={(0.7764895476213607,1.2601525258792372)}] plot[domain=2.60117315331921:3.708321871113299,variable=\t]({1.0*1.049571341072155*cos(\t r)+-0.0*1.049571341072155*sin(\t r)},{0.0*1.049571341072155*cos(\t r)+1.0*1.049571341072155*sin(\t r)});
\draw [shift={(-0.9984068930928022,1.2601525258792372)},dash pattern=on 1pt off 1pt]  plot[domain=-0.5612161566960552:0.5459325610980335,variable=\t]({1.0*1.0495713410721565*cos(\t r)+-0.0*1.0495713410721565*sin(\t r)},{0.0*1.0495713410721565*cos(\t r)+1.0*1.0495713410721565*sin(\t r)});
\draw [shift={(2.324975232147299,2.7063742893553835)}] plot[domain=3.8855943614377084:5.5814743548691155,variable=\t]({1.0*0.6937380625223298*cos(\t r)+-0.0*0.6937380625223298*sin(\t r)},{0.0*0.6937380625223298*cos(\t r)+1.0*0.6937380625223298*sin(\t r)});
\draw [shift={(1.0721071563490672,3.5332672193822168)},dash pattern=on 1pt off 1pt]  plot[domain=4.192395457678183:5.240207521625644,variable=\t]({1.0*1.4942210080003524*cos(\t r)+-0.0*1.4942210080003524*sin(\t r)},{0.0*1.4942210080003524*cos(\t r)+1.0*1.4942210080003524*sin(\t r)});
\draw [shift={(1.0721071563490672,-0.03740679664274582)}] plot[domain=1.2532429248175203:1.8750988121290628,variable=\t]({1.0*2.3860469745561392*cos(\t r)+-0.0*2.3860469745561392*sin(\t r)},{0.0*2.3860469745561392*cos(\t r)+1.0*2.3860469745561392*sin(\t r)});
\draw [shift={(3.6010445686084616,3.5332672193822168)},dash pattern=on 1pt off 1pt]  plot[domain=4.19239545767818:5.240207521625644,variable=\t]({1.0*1.4942210080003553*cos(\t r)+-0.0*1.4942210080003553*sin(\t r)},{0.0*1.4942210080003553*cos(\t r)+1.0*1.4942210080003553*sin(\t r)});
\draw [shift={(3.5778433079455283,1.79245631201308E-4)}] plot[domain=1.2532429248175196:1.875098812129064,variable=\t]({1.0*2.3860469745561357*cos(\t r)+-0.0*2.3860469745561357*sin(\t r)},{0.0*2.3860469745561357*cos(\t r)+1.0*2.3860469745561357*sin(\t r)});
\draw (-0.9991111111111094,1.7496000000000023) node[anchor=north west] {$\mathbf{\delta_2}$};
\draw (5.061111111111119,1.6362000000000023) node[anchor=north west] {$\mathbf{\delta_3}$};
\draw (0.5800000000000031,3.013200000000002) node[anchor=north west] {$\mathbf{\alpha_1}$};
\draw (3.3777777777777835,3.029400000000002) node[anchor=north west] {$\mathbf{\alpha_3}$};
\draw [shift={(-0.18076091944916384,2.3493068877528875)}] plot[domain=4.797326744086996:6.055760546417952,variable=\t]({1.0*0.5469694617899861*cos(\t r)+-0.0*0.5469694617899861*sin(\t r)},{0.0*0.5469694617899861*cos(\t r)+1.0*0.5469694617899861*sin(\t r)});
\draw [shift={(5.016321469047206,2.4808580357117025)}] plot[domain=3.4302948870426033:4.460643866635931,variable=\t]({1.0*0.7260875577864809*cos(\t r)+-0.0*0.7260875577864809*sin(\t r)},{0.0*0.7260875577864809*cos(\t r)+1.0*0.7260875577864809*sin(\t r)});
\draw [shift={(2.2785727108214395,5.9199809037778515)},color=ffxfqq]  plot[domain=4.229640704556637:5.222349592383171,variable=\t]({1.0*4.498158937594167*cos(\t r)+-0.0*4.498158937594167*sin(\t r)},{0.0*4.498158937594167*cos(\t r)+1.0*4.498158937594167*sin(\t r)});
\draw [shift={(2.25537145015851,7.761696975201254)},dash pattern=on 1pt off 1pt,color=ffxfqq]  plot[domain=4.367202708923405:5.069796374239094,variable=\t]({1.0*6.171059260380677*cos(\t r)+-0.0*6.171059260380677*sin(\t r)},{0.0*6.171059260380677*cos(\t r)+1.0*6.171059260380677*sin(\t r)});
\draw [color=ffxfqq](2.5488888888888894,2.203200000000002) node[anchor=north west] {$\mathbf{\gamma_2}$};
\draw [shift={(1.0489058956861375,4.040678790080502)},dash pattern=on 1pt off 1pt,color=ffqqqq]  plot[domain=4.342263909709158:5.142591320835892,variable=\t]({1.0*2.1165988239656426*cos(\t r)+-0.0*2.1165988239656426*sin(\t r)},{0.0*2.1165988239656426*cos(\t r)+1.0*2.1165988239656426*sin(\t r)});
\draw [shift={(6.408397108823017,1.2405186406714517)},dash pattern=on 1pt off 1pt,color=ffqqqq]  plot[domain=2.8021793342196206:3.411432168279142,variable=\t]({1.0*1.9191876947350417*cos(\t r)+-0.0*1.9191876947350417*sin(\t r)},{0.0*1.9191876947350417*cos(\t r)+1.0*1.9191876947350417*sin(\t r)});
\draw [shift={(3.322629440653298,2.5372370991226214)},color=ffqqqq]  plot[domain=3.4542650063588063:5.7412055835026665,variable=\t]({1.0*1.4445679510180998*cos(\t r)+-0.0*1.4445679510180998*sin(\t r)},{0.0*1.4445679510180998*cos(\t r)+1.0*1.4445679510180998*sin(\t r)});
\draw [shift={(3.6666666666666723,6.3018)},color=ffqqqq]  plot[domain=4.036069659596782:4.521828380962042,variable=\t]({1.0*5.4054284068883565*cos(\t r)+-0.0*5.4054284068883565*sin(\t r)},{0.0*5.4054284068883565*cos(\t r)+1.0*5.4054284068883565*sin(\t r)});
\draw [color=ffqqqq](3.6666666666666727,1.2146000000000025) node[anchor=north west] {$\mathbf{\sigma_2}$};
\draw [shift={(9.324450648889117,2.2637510676340726)}] plot[domain=1.5640054032352497:4.705598056825043,variable=\t]({1.0*0.460108682813095*cos(\t r)+-0.0*0.460108682813095*sin(\t r)},{0.0*0.460108682813095*cos(\t r)+1.0*0.460108682813095*sin(\t r)});
\draw [shift={(10.193833358724973,3.2631388863837225)},dash pattern=on 1pt off 1pt]  plot[domain=2.6011731533192077:3.7083218711133004,variable=\t]({1.0*1.0495713410721532*cos(\t r)+-0.0*1.0495713410721532*sin(\t r)},{0.0*1.0495713410721532*cos(\t r)+1.0*1.0495713410721532*sin(\t r)});
\draw [shift={(8.41295222412834,3.2840200209803596)}] plot[domain=-0.5612161566960596:0.5459325610980338,variable=\t]({1.0*1.0495713410721534*cos(\t r)+-0.0*1.0495713410721534*sin(\t r)},{0.0*1.0495713410721534*cos(\t r)+1.0*1.0495713410721534*sin(\t r)});
\draw [shift={(8.44106630546841,1.2649708932202688)}] plot[domain=-0.5612161566960578:0.5459325610980335,variable=\t]({1.0*1.0495713410721526*cos(\t r)+-0.0*1.0495713410721526*sin(\t r)},{0.0*1.0495713410721526*cos(\t r)+1.0*1.0495713410721526*sin(\t r)});
\draw [shift={(10.220602998315151,1.2607946663009417)},dash pattern=on 1pt off 1pt]  plot[domain=2.6011731533192073:3.7083218711133012,variable=\t]({1.0*1.0495713410721492*cos(\t r)+-0.0*1.0495713410721492*sin(\t r)},{0.0*1.0495713410721492*cos(\t r)+1.0*1.0495713410721492*sin(\t r)});
\draw [shift={(5.657876624225276,1.3768881559712978)},dash pattern=on 1pt off 1pt]  plot[domain=-0.5073481922772851:0.48265828210633505,variable=\t]({1.0*1.3537935596819792*cos(\t r)+-0.0*1.3537935596819792*sin(\t r)},{0.0*1.3537935596819792*cos(\t r)+1.0*1.3537935596819792*sin(\t r)});
\draw [shift={(7.978002690518298,1.320509092560377)}] plot[domain=2.59378241582887:3.6720650249256086,variable=\t]({1.0*1.2946773687218227*cos(\t r)+-0.0*1.2946773687218227*sin(\t r)},{0.0*1.2946773687218227*cos(\t r)+1.0*1.2946773687218227*sin(\t r)});
\draw [shift={(5.634675363562344,3.162225163983778)},dash pattern=on 1pt off 1pt]  plot[domain=-0.5073481922772869:0.48265828210633643,variable=\t]({1.0*1.3537935596819781*cos(\t r)+-0.0*1.3537935596819781*sin(\t r)},{0.0*1.3537935596819781*cos(\t r)+1.0*1.3537935596819781*sin(\t r)});
\draw [shift={(7.931600169192435,3.1434321428468044)}] plot[domain=2.5937824158288696:3.6988078867721943,variable=\t]({1.0*1.2946773687218232*cos(\t r)+-0.0*1.2946773687218232*sin(\t r)},{0.0*1.2946773687218232*cos(\t r)+1.0*1.2946773687218232*sin(\t r)});
\draw (9.600000000000012,1.6200000000000023) node[anchor=north west] {$\mathbf{\delta_3}$};
\draw (9.555555555555568,3.677400000000002) node[anchor=north west] {$\mathbf{\delta_4}$};
\draw (5.933333333333342,1.7658000000000023) node[anchor=north west] {$\mathbf{\alpha_2}$};
\draw (5.888888888888897,3.580200000000002) node[anchor=north west] {$\mathbf{\alpha_4}$};
\draw (6.794738396708853,3.819980903777851)-- (9.300474548305317,3.8387739249148254);
\draw (6.841140918034718,0.7191324161772257)-- (9.34687706963118,0.7379254373141992);
\draw [shift={(6.933945960686433,2.241367128272077)}] plot[domain=-1.8483432573008436:1.9277035783602987,variable=\t]({1.0*0.2540309197932255*cos(\t r)+-0.0*0.2540309197932255*sin(\t r)},{0.0*0.2540309197932255*cos(\t r)+1.0*0.2540309197932255*sin(\t r)});
\draw [shift={(3.848178292516714,2.241367128272077)},color=ffxfqq]  plot[domain=-0.3720148408836117:0.39208687853493135,variable=\t]({1.0*4.084392002564145*cos(\t r)+-0.0*4.084392002564145*sin(\t r)},{0.0*4.084392002564145*cos(\t r)+1.0*4.084392002564145*sin(\t r)});
\draw [shift={(0.5071967570547627,2.241367128272077)},dash pattern=on 1pt off 1pt,color=ffxfqq]  plot[domain=-0.20802148819707966:0.21655030497608976,variable=\t]({1.0*7.279727460594187*cos(\t r)+-0.0*7.279727460594187*sin(\t r)},{0.0*7.279727460594187*cos(\t r)+1.0*7.279727460594187*sin(\t r)});
\draw [color=ffxfqq](7.9555555555555655,1.8792000000000022) node[anchor=north west] {$\mathbf{\gamma_3}$};
\draw [color=ffqqqq](8.00000000000001,3.8394000000000017) node[anchor=north west] {$\mathbf{\sigma_3}$};
\draw [shift={(5.472266538921826,1.3393021136973506)},dash pattern=on 1pt off 1pt,color=ffqqqq]  plot[domain=-0.3600665858659511:0.41196859609257785,variable=\t]({1.0*1.760162825939318*cos(\t r)+-0.0*1.760162825939318*sin(\t r)},{0.0*1.760162825939318*cos(\t r)+1.0*1.760162825939318*sin(\t r)});
\draw [shift={(10.135719932170794,3.237397248531672)},dash pattern=on 1pt off 1pt,color=ffqqqq]  plot[domain=2.7192422347009915:3.645578609220703,variable=\t]({1.0*1.317376918499629*cos(\t r)+-0.0*1.317376918499629*sin(\t r)},{0.0*1.317376918499629*cos(\t r)+1.0*1.317376918499629*sin(\t r)});
\draw [shift={(6.5395245294166156,4.440150601297978)},color=ffqqqq]  plot[domain=4.9489988768226185:6.029566738832546,variable=\t]({1.0*2.474449245909201*cos(\t r)+-0.0*2.474449245909201*sin(\t r)},{0.0*2.474449245909201*cos(\t r)+1.0*2.474449245909201*sin(\t r)});
\draw [shift={(9.30047454830531,0.4748231413965702)},color=ffqqqq]  plot[domain=1.7118826245131753:3.0199323189397465,variable=\t]({1.0*2.1449237224564635*cos(\t r)+-0.0*2.1449237224564635*sin(\t r)},{0.0*2.1449237224564635*cos(\t r)+1.0*2.1449237224564635*sin(\t r)});
\draw [shift={(4.266666666666673,5.1678)},color=ffqqqq]  plot[domain=4.333614983460458:4.777059031853636,variable=\t]({1.0*4.49774718210422*cos(\t r)+-0.0*4.49774718210422*sin(\t r)},{0.0*4.49774718210422*cos(\t r)+1.0*4.49774718210422*sin(\t r)});
\end{tikzpicture}
\caption{Supports of two lantern relations}
\end{figure}
%%%%%%%%%%%%%%%%%%%%%%%%%%%%%%%%%%%%%%%%%%%%%%%%%%%%
% 
% prop : [,][,][,][,]\alpha_2^2\alpha_3^2 = 1
%
%%%%%%%%%%%%%%%%%%%%%%%%%%%%%%%%%%%%%%%%%%%%%%%%%%%%%%

\begin{proposition}
Let $f\geq4$ and let $\{\alpha_2,\alpha_3\}$ be any pair of nonseparating simple closed curves 
on $\Sigma_f$ such that $\Sigma_f-\alpha_2-\alpha_3$ is connected.
Then there is a genus $f$ Lefschetz fibration $Z$ over $\Sigma_4$ which has four singular fibers, two of which have monodromy $t_{\alpha_2}$ and another two of which have monodromy $t_{\alpha_3}$.
\begin{proof}
We use the $4$-holed torus relation \cite{KO:2008} and lantern relations. Let $E_2:=t_{\delta_4}^{-1}t_{\delta_3}^{-1}t_{\delta_2}^{-1}t_{\delta_1}^{-1}t_{\alpha_1}t_{\alpha_3}t_{\beta}
t_{\alpha_2}t_{\alpha_4}t_{\beta}t_{\alpha_1}t_{\alpha_3}t_{\beta}t_{\alpha_2}t_{\alpha_4}t_{\beta}$.
We embed the support of this relation into $\Sigma_f$, as shown in Figure $6$. Let $L_5:=t_{\alpha_3}^{-1}t_{\alpha_1}^{-1}t_{\delta_2}^{-1}t_{\delta_3}^{-1}t_{\sigma_2}t_{\alpha_2}t_{\gamma_2}$
and $L_6:=t_{\alpha_4}^{-1}t_{\alpha_2}^{-1}t_{\delta_3}^{-1}t_{\delta_4}^{-1}t_{\sigma_3}t_{\alpha_3}t_{\gamma_3}$.
For the supports of lanterns, see Figure $7$.
Let $w_1:=t_{\beta}t_{\alpha_2}t_{\alpha_4}t_{\beta}t_{\alpha_1}t_{\alpha_3}t_{\beta}t_{\alpha_2}t_{\alpha_4}t_{\beta}$,
$w_2:=t_{\beta}t_{\alpha_1}t_{\alpha_3}t_{\beta}t_{\alpha_2}t_{\alpha_4}t_{\beta}$, and 
$w_3:=t_{\beta}t_{\alpha_2}t_{\alpha_4}t_{\beta}$. Then, from commutativity relations and braid relations, 
\begin{eqnarray*}
1&=&E_2\cdot L_5^{w_1} \cdot L_6^{w_2} \cdot L_5^{w_3} \cdot L_6^{t_{\beta}}\\
&=&t_{\delta_4}^{-1}t_{\delta_3}^{-1}t_{\delta_2}^{-1}t_{\delta_1}^{-1}
(t_{\delta_2}^{-1}t_{\delta_3}^{-1}t_{\sigma_2}t_{\alpha_2}t_{\gamma_2}t_{\beta})(t_{\delta_3}^{-1}t_{\delta_4}^{-1}t_{\sigma_3}t_{\alpha_3}t_{\gamma_3}t_{\beta})\\
&&(t_{\delta_2}^{-1}t_{\delta_3}^{-1}t_{\sigma_2}t_{\alpha_2}t_{\gamma_2}t_{\beta})
(t_{\delta_3}^{-1}t_{\delta_4}^{-1}t_{\sigma_3}t_{\alpha_3}t_{\gamma_3}t_{\beta})\\
&=&(t_{\delta_2}^{-1}t_{\sigma_2}t_{\delta_3}^{-1}t_{\alpha_2}t_{\gamma_2}t_{\delta_1}^{-1}t_{\beta})(t_{\delta_3}^{-1}t_{\sigma_3}t_{\delta_4}^{-1}t_{\alpha_3}t_{\gamma_3}t_{\delta_4}^{-1}t_{\beta})\\
&&(t_{\delta_2}^{-1}t_{\sigma_2}t_{\delta_3}^{-1}t_{\alpha_2}t_{\gamma_2}t_{\delta_2}^{-1}t_{\beta})
(t_{\delta_3}^{-1}t_{\sigma_3}t_{\delta_4}^{-1}t_{\alpha_3}t_{\gamma_3}t_{\delta_3}^{-1}t_{\beta})\\
&=&t_{\alpha_2}(t_{\delta_2}^{-1}t_{t_{\alpha_2}^{-1}(\sigma_2)}t_{\delta_3}^{-1}t_{\gamma_2}t_{\delta_1}^{-1}t_{\beta})t_{\alpha_3}
(t_{\delta_3}^{-1}t_{t_{\alpha_3}^{-1}(\sigma_3)}t_{\delta_4}^{-1}t_{\gamma_3}t_{\delta_4}^{-1}t_{\beta})\\
& &t_{\alpha_2}(t_{\delta_2}^{-1}t_{t_{\alpha_2}^{-1}(\sigma_2)}t_{\delta_3}^{-1}t_{\gamma_2}t_{\delta_2}^{-1}t_{\beta})
t_{\alpha_3}(t_{\delta_3}^{-1}t_{t_{\alpha_3}^{-1}(\sigma_3)}t_{\delta_4}^{-1}t_{\gamma_3}t_{\delta_3}^{-1}t_{\beta})\\
&=&t_{\alpha_2}[t_{\delta_2}^{-1}t_{t_{\alpha_2}^{-1}(\sigma_2)}t_{\delta_3}^{-1},\phi_1]
t_{\alpha_3}[t_{\delta_3}^{-1}t_{t_{\alpha_3}^{-1}(\sigma_3)}t_{\delta_4}^{-1},\phi_2]\\
& &t_{\alpha_2}[t_{\delta_2}^{-1}t_{t_{\alpha_2}^{-1}(\sigma_2)}t_{\delta_3}^{-1},\phi_3]t_{\alpha_3}[t_{\delta_3}^{-1}t_{t_{\alpha_3}^{-1}(\sigma_3)}t_{\delta_4}^{-1},\phi_4]\\
&=&[t_{\delta_2}^{-1}t_{t_{\alpha_2}^{-1}(\sigma_2)}t_{\delta_3}^{-1},\phi_1]^{t_{\alpha_2}^{-1}}
[t_{\delta_3}^{-1}t_{t_{\alpha_3}^{-1}(\sigma_3)}t_{\delta_4}^{-1},\phi_2]^{t_{\alpha_3}^{-1}t_{\alpha_2}^{-1}}\\
& &[t_{\delta_2}^{-1}t_{t_{\alpha_2}^{-1}(\sigma_2)}t_{\delta_3}^{-1},\phi_3]^{t_{\alpha_3}^{-1}t_{\alpha_2}^{-2}}
[t_{\delta_3}^{-1}t_{t_{\alpha_3}^{-1}(\sigma_3)}t_{\delta_4}^{-1},\phi_4]^{t_{\alpha_3}^{-2}t_{\alpha_2}^{-2}}t_{\alpha_2}^2t_{\alpha_3}^2
\end{eqnarray*}
For the fifth equality, we need to find certain $\phi_1,\phi_2,\phi_3$ and $\phi_4$. For $\phi_1$, it is sufficient to verify that $\{\delta_2,t_{\alpha_2}^{-1}(\sigma_2),\delta_3\}$ is topologically equivalent to $\{\beta,\delta_1,\gamma_2\}$. This is because $\{\delta_2,t_{\alpha_2}^{-1}(\sigma_2),\delta_3\}$ is topologically equivalent to $\{\delta_2,\beta,\delta_3\}$, and then $\{\delta_2,\beta,\delta_3\}$ to $\{\beta,\delta_1,\gamma_2\}$. The arguments for $\phi_2,\phi_3,$ and $\phi_4$ are similar. For these, we can check that $\{\delta_3,t_{\alpha_3}^{-1}(\sigma_3),\delta_4\}$ is topologically equivalent to $\{\beta,\delta_4,\gamma_3\}$, $\{\delta_2,t_{\alpha_2}^{-1}(\sigma_2),\delta_3\}$ is topologically equivalent to $\{\beta,\delta_2,\gamma_2\}$, and $\{\delta_3,t_{\alpha_3}^{-1}(\sigma_3),\delta_4\}$ is topologically equivalent to $\{\beta,\delta_3,\gamma_3\}$. 

\end{proof}
\end{proposition}

%%%%%%%%%%%%%%%%%%%%%%%%%%%%%%%%%%%%%%%%%%%%%%%%%%%%%%%%%
%
% prop : [,][,][,]\beta\gamma = 1
%
%%%%%%%%%%%%%%%%%%%%%%%%%%%%%%%%%%%%%%%%%%%%%%%%%%%%%%%%%%
\begin{proposition}
Let $f\geq 6$ and let $\beta$, $\gamma$ be simple closed curves 
on $\Sigma_f$ embedded, as shown in Figure $2$.
Then there is a genus $f$ Lefschetz fibration $W$ over $\Sigma_3$ which has two singular fibers, one of which has monodromy $t_{\beta}$ and another has monodromy $t_{\gamma}$.

\begin{proof}
There is a $9$-holed torus relation  $E_7:=t_{\delta_1}^{-1}t_{\delta_2}^{-1}\ldots t_{\delta_8}^{-1}t_{\gamma_9}^{-1}t_{\beta_8}t_{\sigma_3}t_{\sigma_6}t_{\alpha_{10}}t_{\beta_5}$
$t_{\sigma_4}t_{\sigma_7}t_{\alpha_6}t_{\beta_2}t_{\sigma_5}t_{\sigma_8}t_{\alpha_3}$ 
(see its support in orange in Figure $8$ and see Figure $9$ for its interior curves),
where we use the idenfication $(\alpha_1,\alpha_2,\alpha_3,\alpha_4,\alpha_5,\alpha_6,\alpha_7,\alpha_8,$ $\alpha_9)\rightarrow (\alpha_5,\alpha_6,\alpha_7,\alpha_8,\alpha_{10},\alpha_1,\alpha_2,\alpha_3,\alpha_4)$ to go from Figure $9$ in \cite{KO:2008} to Figure $9$ in this article. Here, each $\beta_i=t_{\alpha_i}(\beta)$ as in \cite{KO:2008}. If we combine this relation $E_7$ and one more lantern relation $L_8:=t_{\delta_9}^{-1}t_{\delta_{10}}^{-1}t_{\gamma_9}t_{\sigma_9}t_{\alpha_9}t_{\alpha_8}^{-1}t_{\alpha_{10}}^{-1}$ (see its support in blue in Figure $8$), 
then we get the following $10$-holed torus relation
$E_8:=t_{\delta_1}^{-1}t_{\delta_2}^{-1}\ldots t_{\delta_{10}}^{-1}t_{\alpha_8}^{-1}t_{\alpha_{10}}^{-1}
t_{\beta_8}t_{\sigma_3}t_{\sigma_6}t_{\alpha_{10}}t_{\beta_5}t_{\sigma_4}t_{\sigma_7}t_{\alpha_6}t_{\beta_2}t_{\sigma_5}t_{\sigma_8}t_{\alpha_3}t_{\sigma_9}t_{\alpha_9}$. 
Let $\beta_5'=(t_{\sigma_4}t_{\sigma_7}t_{\alpha_6}t_{\sigma_5}t_{\sigma_8}t_{\alpha_3}t_{\sigma_9}t_{\alpha_9})^{-1}(\beta_5)$ and 
$\beta_2'=(t_{\sigma_5}t_{\sigma_8}t_{\alpha_3}t_{\sigma_9}t_{\alpha_9})^{-1}(\beta_2)$. 
Then, by using commutativity relations and braid relations, 
$$1=t_{\delta_1}^{-1}t_{\delta_2}^{-1}\ldots t_{\delta_{10}}^{-1}
t_{\beta_8}t_{\sigma_3}t_{\sigma_6}t_{\alpha_{10}}t_{\sigma_4}t_{\sigma_7}t_{\alpha_6}t_{\sigma_5}t_{\sigma_8}t_{\alpha_3}t_{\sigma_9}t_{\alpha_9}
t_{\beta_5'}t_{\alpha_8}^{-1}t_{t_{\alpha_8}(\beta_2')}t_{\alpha_{10}}^{-1}$$
$$=t_{\delta_1}^{-1}t_{\delta_2}^{-1}\ldots t_{\delta_{10}}^{-1}
t_{\beta_8}t_{\sigma_3}t_{\sigma_4}t_{\sigma_5}t_{\sigma_6}t_{\sigma_7}t_{\sigma_8}
t_{\alpha_{10}}t_{\alpha_6}t_{\alpha_3}t_{\alpha_9}t_{t_{\alpha_9}^{-1}(\sigma_9)}
t_{\beta_5'}t_{\alpha_8}^{-1}t_{t_{\alpha_8}(\beta_2')}t_{\alpha_{10}}^{-1}$$
$$=t_{\delta_1}^{-1}t_{\delta_2}^{-1}\ldots t_{\delta_{10}}^{-1}
(t_{\sigma_3}t_{\sigma_4}t_{\sigma_5}t_{\sigma_6}t_{\sigma_7}t_{\sigma_8}
t_{\alpha_{10}}t_{\alpha_6}t_{\alpha_3}t_{\alpha_9})^{t_{\beta_8}^{-1}}t_{\beta_8}t_{t_{\alpha_9}^{-1}(\sigma_9)}
t_{\beta_5'}t_{\alpha_8}^{-1}t_{t_{\alpha_8}(\beta_2')}t_{\alpha_{10}}^{-1}$$   
$$=\{t_{\delta_1}^{-1}\cdot t_{t_{\beta_8}(\sigma_3)}\cdot t_{\delta_3}^{-1}\cdot t_{t_{\beta_8}(\sigma_4)}
\cdot t_{\delta_{10}}^{-1}\cdot t_{t_{\beta_8}(\sigma_5)}\cdot t_{\delta_2}^{-1}\cdot t_{t_{\beta_8}(\sigma_6)}\cdot t_{\delta_7}^{-1}
\cdot t_{t_{\beta_8}(\sigma_7)}\cdot t_{\delta_9}^{-1}\cdot t_{t_{\beta_8}(\sigma_8)}\}$$
$$\{t_{\delta_5}^{-1}\cdot t_{t_{\beta_8}(\alpha_3)}\cdot t_{\delta_8}^{-1}\cdot t_{t_{\beta_8}(\alpha_{10})}
\cdot t_{\delta_6}^{-1}\cdot t_{t_{\beta_8}(\alpha_6)}\cdot t_{\delta_4}^{-1}\cdot t_{t_{\beta_8}(\alpha_9)}\}
\{t_{\beta_8}\cdot t_{t_{\alpha_9}^{-1}(\sigma_9)}\}\{t_{\beta_5'}\cdot t_{\alpha_8}^{-1}\cdot t_{t_{\alpha_8}(\beta_2')}\cdot t_{\alpha_{10}}^{-1}\}$$
$$=[t_{\delta_1}^{-1}\cdot t_{t_{\beta_8}(\sigma_3)}\cdot t_{\delta_3}^{-1}\cdot t_{t_{\beta_8}(\sigma_4)}
\cdot t_{\delta_{10}}^{-1}\cdot t_{t_{\beta_8}(\sigma_5)},\phi_1][t_{\delta_5}^{-1}\cdot t_{t_{\beta_8}(\alpha_3)}\cdot t_{\delta_8}^{-1}\cdot t_{t_{\beta_8}(\alpha_{10})},\phi_2]$$
$$\cdot t_{\beta_8}\cdot t_{t_{\alpha_9}^{-1}(\sigma_9)}\cdot[t_{\beta_5'}t_{\alpha_8}^{-1},\phi_3]$$

For the last equality, we need to verify that $\{\delta_1,t_{\beta_8}(\sigma_3),\delta_3,t_{\beta_8}(\sigma_4),\delta_{10},t_{\beta_8}(\sigma_5)\}$ is topologically equivalent to 
$\{t_{\beta_8}(\sigma_8),\delta_9,t_{\beta_8}(\sigma_7),\delta_7,t_{\beta_8}(\sigma_6),\delta_2\}$. This follows from the fact that both $\Sigma_f-\delta_1-\delta_3-\delta_{10}-\sigma_3-\sigma_4-\sigma_5$ and 
$\Sigma_f-\delta_2-\delta_7-\delta_9-\sigma_6-\sigma_7-\sigma_8$ are connected. For $\phi_2$ and $\phi_3$, it is easy to check that $\Sigma_f-\delta_5-\alpha_3-\delta_8-\alpha_{10}\approx \Sigma_{f-4}^8\approx \Sigma_f-\alpha_9-\delta_4-\alpha_6-\delta_6$ and that $\{\beta_5',\alpha_8\}$ is topologically equivalent to $\{\beta,\alpha_8\}$ and $\{\alpha_{10},t_{\alpha_8}(\beta_2')\}$ is topologically equivalent to $\{\alpha_{10},\beta\}$. Finally, observe that $\{\beta_8,t_{\alpha_9}^{-1}(\sigma_9)\}$ is topologically equivalent to $\{\beta,t_{\alpha_9}^{-1}(\sigma_9)\}$ and $t_{\alpha_9}^{-1}(\sigma_9)=\gamma$.

\end{proof}

\end{proposition}
%%%%%%%%%%%%%%%%%%%%%%%%%%%%%%%%%5
%
% figure 8 & 9 : 10-holed
%
%%%%%%%%%%%%%%%%%%%%%%%%%%%%%%%%
\begin{figure}

\definecolor{qqqqff}{rgb}{0.0,0.0,1.0}
\definecolor{ffxfqq}{rgb}{1.0,0.4980392156862745,0.0}

\caption{Supports for a $9$-holed torus relation and a lantern relation and their embeddings into a genus $6$ surface}

%%%%%%%%%%%%%%%%%%%%%%%%%%%%%%%%%%%%%%%%%%%%%%%%%%%%%%%%%%%%
\definecolor{qqqqff}{rgb}{0.0,0.0,1.0}
\definecolor{ffxfqq}{rgb}{1.0,0.4980392156862745,0.0}
\definecolor{ffqqqq}{rgb}{1.0,0.0,0.0}
%

\caption{Interior curves for a $10$-holed torus relation}
\end{figure}

\end{section}

%%%%%%%%%%%%%%%%%%%%%%%%%%%%%%%%%%
%
% section 4: signature computation
%
%%%%%%%%%%%%%%%%%%%%%%%%%%%%%%%%%%

\begin{section}{Signature computation}

In order to compute the signature of the total space of surface bundles, we first review the definition of Meyer's signature cocycle.

\begin{definition}
For any given $A,B\in Sp(2g,\mathbb{R})$, consider the subspace 
$$V_{A,B}:=\{(x,y)\in \mathbb{R}^{2g}\times \mathbb{R}^{2g}|(A^{-1}-I_{2g})x+(B-I_{2g})y=0\}$$ of the real vector space $\mathbb{R}^{2g}\times \mathbb{R}^{2g}$ and the 
bilinear form $<,>_{A,B}:(\mathbb{R}^{2g}\times \mathbb{R}^{2g})\times (\mathbb{R}^{2g}\times \mathbb{R}^{2g})\rightarrow \mathbb{R}$ defined by $<(x_1,y_1),(x_2,y_2)>_{A,B}:=(x_1+y_1)\cdot J(I_{2g}-B)y_2$,
where $\cdot$ is the inner product of $\mathbb{R}^{2g}$ and $J$ is the matrix representing the multiplication by $\sqrt{-1}$ on $\mathbb{R}^{2g}=\mathbb{C}^{g}$. Since the restriction of $<,>_{A,B}$ on $V_{A,B}$ is symmetric, we can define
$\tau_g(A,B):=sign(<,>_{A,B},V_{A,B})$.
\end{definition}

We denote by $\psi:Mod(\Sigma_g)\rightarrow Sp(2g:\mathbb{R})$ the symplectic representation of the mapping class group. 

\begin{theorem}\cite{Mey:1972}
Let $E_{A,B}\rightarrow P$ be an oriented $\Sigma_g$ bundle over a pair of pants $P$ whose monodromy representation $\chi$ composed with the symplectic representation $\psi$ is given by $\psi\circ\chi:\pi_1(P,*)\rightarrow Sp(2g:\mathbb{R})$ sending one generator to $A$ and the other to $B$. Then $\sigma(E_{A,B})=-\tau_g(A,B)$.   
\end{theorem}

We can easily check that $\tau_g$ is a $2$-cocycle on the symplectic group $Sp(2g,\mathbb{R})$ using Novikov's additivity. We call this $\tau_g$ Meyer's signature cocycle.
The pants decomposition of any base surface gives the following signature formula.

\begin{theorem}\cite{Mey:1973}
Let $f:E\rightarrow \Sigma_h^{r}$ be an oriented surface bundle with fiber $\Sigma_g$ and monodromy representation $\chi: \pi_1(\Sigma_h^{r})\rightarrow Mod(\Sigma_g)$. Fix a standard presentation of $\pi_1(\Sigma_h^{r})$ as follows: 
$$\pi_1(\Sigma_h^{r})=<a_1,b_1,\cdots,a_h,b_h,c_1,\cdots,c_r|\prod_{i=1}^{h}[a_i,b_i]\prod_{j=1}^{r}c_j=1>$$
and let $\tau_g$ be Meyer's signature cocycle.
Then the signature of $E$ is given by the formula
$$\sigma(E)=\sum_{i=1}^{h}\tau_g(\kappa_i,\beta_i)-\sum_{i=2}^{h}\tau_g(\kappa_1\cdots\kappa_{i-1},\kappa_i)-\sum_{j=1}^{r-1}\tau_g(\kappa_1\cdots\kappa_h\gamma_1\cdots\gamma_{j-1},\gamma_j)$$
where $\alpha_i=\psi\circ\chi(a_i),\beta_i=\psi\circ\chi(b_i),\gamma_i=\psi\circ\chi(c_i)$ and
$\kappa_i=[\alpha_i,\beta_i]$.
\end{theorem}

By applying this formula, we can compute the signatures of surface bundles obtained by taking out some neighborhoods of singular fibers from the Lefschetz fibrations constructed in Section $3$. We used Mathematica for computing each term in the above formula.

Meyer also provided another interpretation of the above signature formula. For this, we start with the following diagram. 

$$
\begin{array}{ccccccccc}
1 & \rightarrow & \widetilde{R} & \rightarrow & \widetilde{F} & \xrightarrow{\widetilde{\pi}} & \pi_1(\Sigma_h) & \rightarrow & 1 \\
  &  &  \downarrow &  & \downarrow &  & \downarrow \chi &  &  \\
1 & \rightarrow & R & \rightarrow & F & \xrightarrow{\pi} & Mod(\Sigma_g) & \rightarrow & 1 \\
\end{array}
$$
Here, $\pi_1(\Sigma_h)=<a_1,\cdots,a_h,b_1,\cdots,b_h|\prod_{i=1}^{h}[a_i,b_i]=1>$, $\widetilde{F}=<\widetilde{a_1},\cdots,\widetilde{a_h},$
$\widetilde{b_1},\cdots,\widetilde{b_h}>$, $\widetilde{R}$ is the normal closure of $\widetilde{r}=\prod_{i=1}^{h}[\widetilde{a_i},\widetilde{b_i}]$, and $\widetilde{\pi}:\widetilde{a_i}\mapsto a_i, \widetilde{b_i}\mapsto b_i$.
The second row corresponds to the finite presentation of $Mod(\Sigma_g)$ due to Wajnryb. $F=F(S)$, where $S=\{y_1,y_2,u_1,\cdots,u_g,z_1,\cdots,z_{g-1}\}$ and $R$ is the normal closure of $A^{k}_{i,j}$'s, $B^{k}_{i}$'s, $C^1, D^1, E^1$ (cf.\cite{Endo:1998} \S$3$). If we have a monodromy representation $\chi:\pi_1(\Sigma_h)\rightarrow Mod(\Sigma_g)$, then
there exists a homomorphism $\widetilde{\chi}:\widetilde{F}\rightarrow F$ such that $\chi\circ\widetilde{\pi}=\pi\circ\widetilde{\chi}$ since $\pi$ is surjective and $\widetilde{F}$ is free. Hence we have $\widetilde{\chi}(\widetilde{r})\in R\cap[F,F]$.
Now define the $1$-cochain $c:F\rightarrow \mathbb{Z}$ cobounding the $2$-cocycle $-\pi^*\psi^*(\tau_g)$ as follows. 
$$
c(x):=\sum_{j=1}^{m}\tau_g(\psi(\pi(\widetilde{x_{j-1}})),\psi(\pi(x_j)))
$$
$$
(x=\prod_{j=1}^{m}x_j,\quad x_j\in S\cup S^{-1}, \quad\widetilde{x_j}=\prod_{i=1}^{j}x_i)
$$
Since $\pi^*\psi^*(\tau_g)\mid_{R\times R}=0$, the restriction $c\mid_{R}:R\rightarrow \mathbb{Z}$ is a homomorphism.
The values of $c$ for the relations of Wajnryb's presentation were calculated in \cite{Endo:1998}.
\begin{theorem}\cite{Mey:1973}

Let $p:E\rightarrow \Sigma_h$ be a $\Sigma_g$- bundle over $\Sigma_h$ 
and $\chi:\pi_1(\Sigma_h)\rightarrow Mod(\Sigma_g)$ be its monodromy homomorphism. Then the signature of the total space $E$ is given as follows :
$$
\sigma(E) = -c\mid_R(\widetilde{\chi}(\widetilde{r}))      \quad( =  <\psi^*[\tau_g], \widetilde{\chi}(\widetilde{r})[R,F]> )
$$
where $<,>$ is a pairing on the second cohomology and homology of $Mod(\Sigma_g)$.
\end{theorem}

Now, we are ready to prove our main theorem. 
%%%%%%%%%%%%%%%%%%%%%%%%%%%%%%%%%%%
%   
%   proof of main theorem
%
%%%%%%%%%%%%%%%%%%%%%%%%%%%%%%%%%%%%%

\begin{proof}[Proof of Theorem \ref{thm:main}]

(a)  We apply the subtraction operation to the Lefschetz fibrations $X\rightarrow \Sigma_3$, $Y_1\rightarrow \Sigma_2$, and $Y_2\rightarrow \Sigma_3$ constructed in Propositions $3.5$ and Proposition $3.1$. Let $N_1\subset X$ be the neighborhood of four singular fibers with coinciding vanishing cycles and $N_2\subset X$ be the neighborhood of two singular fibers with coinciding vanishing cycles. Then the complement $X\setminus N_1\setminus N_2$ is the $\Sigma_f$ bundle over $\Sigma_3^{2}$, and its signature can be computed by applying Theorem $4.2$ to this bundle. More precisely to its monodromy representation $\chi:\pi_1(\Sigma_3^2)\rightarrow Mod(\Sigma_f)$ given by $\chi(a_1)=(t_{\delta_1}^{-1}\cdot t_{t_{\alpha_1}^{-1}(\sigma_1)}\cdot t_{\delta_2}^{-1})^{t_{\alpha_1}^{-2}}$, $\chi(b_1)=(\phi_1)^{t_{\alpha_1}^{-2}}$, $\chi(a_2)=(t_{\delta_1}^{-1}\cdot t_{t_{\alpha_1}^{-1}(\sigma_1)}\cdot t_{\delta_2}^{-1})^{t_{\alpha_1}^{-4}}$, $\chi(b_2)=(\phi_2)^{t_{\alpha_1}^{-4}}$, $\chi(a_3)=(t_{\delta_2}^{-1}\cdot t_{t_{\alpha_2}^{-1}(\sigma_2)}\cdot t_{\delta_3}^{-1})^{t_{\alpha_1}^{-4} t_{\alpha_2}^{-2}}$, $\chi(b_3)=(\phi_3)^{t_{\alpha_1}^{-4} t_{\alpha_2}^{-2}}$, $\chi(c_1)=t_{\alpha_1}^4$, and $\chi(c_2)=t_{\alpha_2}^2$. Now, by computations using Mathematica we have $\tau(\kappa_1,\beta_1)=\tau(\kappa_2,\beta_2)=\tau(\kappa_3,\beta_3)=2$, $-\tau(\kappa_1,\kappa_2)=-\tau(\kappa_1\kappa_2,\kappa_3)=-2$, and $-\tau(\kappa_1\kappa_2\kappa_3,\gamma_1)=0$. Hence, $\sigma(X\setminus N_1\setminus N_2)=3\cdot2+2\cdot(-2)+0=2$. By taking out the neighborhood $M_i$ of all singular fibers from $Y_i$ (for $i=1,2$), we get $Y_i\setminus M_i$, the $\Sigma_f$ bundles over ${\Sigma}^{1}_2$, both with signature $-1$. For signature computation, we can directly apply Theorem $4.2$ to these two bundles as above. Alternatively, we can first compute the signature of Lefschetz fibrations : $\sigma(Y_1)=-2$ and $\sigma(Y_2)=-4$ (cf. Proposition $15$ and Proposition $16$ in \cite{EKKOS:2002}). In order to compute the signature of taken out parts, apply Theorem $4.1$ several times and use the fact that $\sigma(N($a nonseparating singular fiber$))=0$ (cf.\cite{Oz:2002}). From these, we have $\sigma(Y_1\setminus M_1)=(-2)-(-1)=-1$ and $\sigma(Y_2\setminus M_2)=(-4)-(-3)=-1$. Therefore, $X-Y_1-Y_2$ is the $\Sigma_{f\geq3}$ bundle over $\Sigma_8$, and $\sigma(X-Y_1-Y_2)=\sigma(X\setminus N_1\setminus N_2)+\sigma(\overline{Y_1\setminus M_1})+\sigma(\overline{Y_2\setminus M_2})=2+1+1=4$ by Novikov additivity. Moreover, if we pullback this bundle (or, with opposite orientation) to unramified coverings of $\Sigma_8$ of degree $|n|$, then we get $b(f\geq3,n)\leq7|n|+1$.

(b) Apply the subtraction operation to the Lefschetz fibrations $Z\rightarrow \Sigma_4$ and $Y_3\rightarrow \Sigma_3$, constructed in Proposition $3.6$ and Proposition $3.3$, respectively. Then, $Z-Y_3$ is the required $\Sigma_{f\geq5}$ bundle over $\Sigma_7$. Let $N$ be the neighborhood of all singular fibers in $Z$ and let $M$ be the neighborhood of all singular fibers in $Y_3$. By applying Theorem $4.2$ to two surface bundles $Z\setminus N$ and $Y_3\setminus M$, we get $\sigma(Z-Y_3)=\sigma(Z\setminus N)+\sigma(\overline{Y_3\setminus M})=2+2=4$. Let me give you another proof for verifying $\sigma(Z-Y_3)=4$ using Theorem $4.3$. 
From Proposition $3.6$ and Proposition $3.3$, we have
$\widetilde{\chi}(\widetilde{r})\equiv(E_2\cdot L_5^{w_1}\cdot L_6^{w_2}\cdot L_5^{w_3}\cdot L_6^{t_\beta})(L_1\cdot L_2^{t_y t_x t_z}\cdot L_3^{t_z}\cdot L_4)^{g}$ modulo commutativity and braid relations, where $g$ is a self-homeomorphism of $\Sigma_{f\geq5}$ such that $g(\alpha_3)=b$ and $g(\alpha_2)=c$. Moreover, from \cite{KO:2008}, $E_2\equiv L_{10}\cdot (L_{9}\cdot ((C^1)^{-1})^{z_0})^{z_1}$ for some mapping classes $z_0, z_1$, modulo commutativity and braid relations. Observe that for each $L_i$, four boundary curves are nonseparating and $\Sigma_f\setminus supp(L_i)$ is connected. Since the same holds for the relation $(D^1)^{-1}$, there exists a self-homeomorphism $f_i$ of $\Sigma_f$ sending the supp($(D^1)^{-1}$) to the supp($L_i$) for each $i$. Therefore, 
$\widetilde{\chi}(\widetilde{r})\equiv((D^1)^{-1})^{f_{10}}
((D^1)^{-1})^{f_9\circ z_1}\cdot ((C^1)^{-1})^{z_0\circ z_1}\cdot ((D^1)^{-1})^{f_5\circ w_1}\cdot ((D^1)^{-1})^{f_6\circ w_2}\cdot ((D^1)^{-1})^{f_5\circ w_3}\cdot ((D^1)^{-1})^{f_6\circ t_\beta}\cdot ((D^1)^{-1})^{f_1\circ g}\cdot ((D^1)^{-1})^{f_2\circ (t_y t_x t_z)\circ g}\cdot ((D^1)^{-1})^{f_3\circ t_z\circ g}\cdot ((D^1)^{-1})^{f_4\circ g}$ modulo commutativity and braid relations and hence $\sigma(Z-Y_3) = -c(\widetilde{\chi}(\widetilde{r}))=c(C^1)+10\cdot c(D^1)=(-6)+10=4$.
For the upper bound for the genus function $b(f\geq5,n)$, use the same argument as before. 

(c) Apply the subtraction operation to the Lefschetz fibrations $W\rightarrow \Sigma_3$ and $Y_4\rightarrow \Sigma_3$, constructed in Proposition $3.7$ and Proposition $3.4$, respectively. Then $W-Y_4$ is the required $\Sigma_{f\geq6}$ bundle over $\Sigma_6$ with signature $4$.
From Proposition $3.4$ and Proposition $3.7$, $\widetilde{\chi}(\widetilde{r})\equiv E_8\cdot (L_1\cdot L_2)^{h}$ modulo braid and commutativity relations, where $h$ is a self-homeomorphism of $\Sigma_f$ such that $h\{\beta_8,t_{\alpha_9}^{-1}(\sigma_9)\}=\{\beta,\gamma\}$. Moreover, $E_8\equiv (\prod_{j=1}^{8} ((D^1)^{-1})^{z_{j}})\cdot ((C^1)^{-1})^{z_0}$ for some $z_0,\ldots, z_8$ (cf.\cite{KO:2008} and Proposition $3.7$).
Therefore, $\sigma(W-Y_4)=-c(\widetilde{\chi}(\widetilde{r}))=c(C^1)+10\cdot c(D^1)=(-6)+10=4$.
For the upper bound for the genus function $b(f\geq6,n)$, use the same argument as before.

\end{proof}

%%%%%%%%%%%%%%%%%%%%%%%%
%
% proof of thm 1.4
%
%%%%%%%%%%%%%%%%%%%%%%%%%%

\begin{proof}[Proof of Theorem \ref{thm:cor}]
Every odd genus surface is a covering of genus three surface. By Morita\cite{Morita:1987}, after replacing a given surface bundle by a pullback to some covering of the base, the resulting surface bundle admits a fiberwise covering of any given degree. After applying this to the genus $3$ surface bundle over $\Sigma_{b_3(1)}$ with signature $4$ and the degree of the covering $\Sigma_f\rightarrow \Sigma_3$, we obtain $b_f(\frac{f-1}{2}n)\leq n(b_3(1)-1)+1$. Hence, $G_f:=\lim_{n\to\infty}\frac{b_f(n)}{n}\leq \lim_{n\to\infty}\frac{2n(b_3(1)-1)+2}{(f-1)n}\leq \lim_{n\to\infty}\frac{14n+2}{(f-1)n}=\frac{14}{f-1}$.  
\end{proof}

%%%%%%%%%%%%%%%%%%%%%
%
% Final remark
% 
%%%%%%%%%%%%%%%%%%%%%

\begin{remark}
In  \cite{Harer:1983, KorkSt:2003, Sa:2012}, it was proven that $H_2(Mod(\Sigma_g):\mathbb{Z})\cong\mathbb{Z}$ for every $g\geq 4$ and $H_2(Mod(\Sigma_g):\mathbb{Z})\cong\mathbb{Z}\oplus \mathbb{Z}_2$ for $g=3$. Meyer\cite{Mey:1973} proved that each generator of $H_2(Mod(\Sigma_g))/Tor$ gives us signature $4$ relying on the Theorem $4.3$. In order to prove this, Meyer used Birman-Hilden's presentation of $Mod(\Sigma_g)$, and Endo\cite{Endo:1998} reproved this using a simpler presentation due to Wajnryb\cite{Waj:1999}. By taking $\widetilde{\chi}(\widetilde{r})$ as different representatives for a generator of $H_2(Mod(\Sigma_g))/Tor$,
we can construct various surface bundles with a fixed signature $4$ as we have seen in the proof of Theorem $1.2$. Therefore, the problem to determine $b(f,n)$ is to find the most effective representative $\tilde{\chi}(\tilde{r})$, in the sense of commutator length, for $n$ times generator of $H_2(Mod(\Sigma_f))/Tor$.
\end{remark}
\end{section}

\subsection*{Acknowledgements}
The author would like to thank her advisor, Jongil Park, for suggesting this problem, providing various opportunities to learn and discuss math, and his advice and encouragement. The author would also like to thank Ki-Heon Yun for his seminar talks which provided helpful background knowledge on this problem. The author owes special gratitude to Hisaaki Endo for his intensive lectures at KIAS, which provided crucial ideas of signature computation, and Mustafa Korkmaz for 
kindly answering the question about his computation used in \cite{BDS:2001}. Part of this work was carried out during a visit to the Max Planck Institute in Bonn in $2013$. The author thanks the institute for its hospitality and the organizers of the special semester of $4$-manifold for the invitation. She also would like to thank the referee for the careful review and valuable comments. This work was supported by the National Research Foundation of Korea Grant(2010-0019516).


\begin{thebibliography}{99}

\bibitem{Atiyah:69} M.F. Atiyah, {\it The signature of fiber-bundles}, 1969 Global Analysis (papers in Honor of K.Kodaira), 73--84
\MR{0254864 (40 \#8071)}


\bibitem{BD:2002}
J. Bryan, R. Donagi, {\it Surface bundles over surfaces of small genus}, Geom. Topol. \textbf{6} (2002), 59--67 \MR{1885589 (2002m:14037)}


\bibitem{BDS:2001}
J. Bryan, R. Donagi, and A. Stipsicz, {\it Surface bundles: some interesting examples}, Turkish J. Math. \textbf{25} (2001), no.~1, 61--68 \MR{1829079 (2002a:57034)}

\bibitem{BM:2013}
R.I. Baykur and D. Margalit, {\it Indecomposable surface bundles over surfaces}, J. Topol. Anal. \textbf{5} (2013), no. 1, 161--181 \MR{3062948}

\bibitem{Br:1982}
K.S. Brown, {\it Cohomology of Groups}, Graduate Texts in Mathematics \textbf{87}, Springer-Verlag, 1982. \MR{83k:20002} 

\bibitem{EKKOS:2002}
H. Endo, M. Korkmaz, D. Kotschick, B. Ozbagci and A. Stipsicz, {\it Commutators, Lefschetz fibrations and the signatures of surface bundles}, Topology \textbf{41} (2002), no. 5, 961–-977 \MR{1923994 (2003f:57051)}

\bibitem{EMV:2011}
H. Endo, T.E. Mark, and J. Van Horn-Morris, {\it Monodromy substitutions and rational blowdowns}, J. Topol. \textbf{4} (2011), no. 1, 227--253 \MR{2783383 (2012b:57051)}

\bibitem{Endo:1998}
H. Endo, {\it A construction of surface bundles over surfaces with non-zero signature}, Osaka J. Math. \textbf{35} (1998), no.~4, 915-–930. \MR{1659565 (99m:55017)}


\bibitem{Endo:2000}
H. Endo, {\it Meyer's signature cocycle and hyperelliptic fibrations}, Math. Ann. \textbf{316} (2000), no.~2, 237--257. \MR{1741270 (2001b:57047)}


\bibitem{EN:2006}
H.~Endo and S.~Nagami {\it Signature of relations in mapping class groups and non-holomorphic Lefschetz fibrations}, Trans. Amer. Math. Soc. \textbf{357} (2005), no. 8, 3179–-3199. \MR{2135741 (2006g:57051)}


\bibitem{FM:2012}
B. Farb and D. Margalit, {\it A primer on mapping class groups}, 
Princeton Mathematical Series \textbf{49} Princeton University Press,
Princeton, NJ, 2012


\bibitem{Hamada:2014}
N. Hamada, {\it Upper bounds for the minimal number of singular fibers in a Lefschetz fibration over the torus}, Michigan Math. J. \textbf{63} (2014), no. 2, 275--291. \MR{3215551} 

\bibitem{Hamen:2012}
U. Hamenstadt, {\it Signatures of Surface bundles and Milnor Wood Inequalities}, preprint, http://arxiv.org/pdf/1206.0263

\bibitem{Harer:1983}
J. Harer, {\it The second homology group of the mapping class group of an orientable surface}, Invent. Math. \textbf{72} (1983), 221--239. \MR{0700769 (84g:57006)}

\bibitem{Hi:1969}
F. Hirzebruch, {\it The signature of ramified coverings}, Global Analysis (Papers in Honor of K. Kodaira), Univ.Tokyo Press,Tokyo (1969), 253--265 \MR{0258060 (41 \#2707)}

\bibitem{Hoster:2001}
M. Hoster, {\it A new proof of the signature formula for surface bundles}, Topology Appl. \textbf{112} (2001), no. 2, 205--213.
\MR{1823605 (2002a:57038)}

\bibitem{Kirby}
R. Kirby, {\it Problems in Low-Dimensional Topology} in W. Kazez(Ed.), Geometric Topology, AMS/IP Studies in Advanced Mathematics, Vol.2.2, American Mathematical Society, Providence, RI, 1997.

\bibitem{K:2004}
M. Korkmaz, {\it Stable Commutator Length of a Dehn Twist}, Michigan Math. J. \textbf{52} (2004), no. 1, 23-–31. \MR{2043394 (2005b:57003)}

\bibitem{KO:2001}
M. Korkmaz, B. Ozbagci, {\it Minimal number of singular fibers in a Lefschetz fibration}, Proc. Amer. Math. Soc. \textbf{129} (2001), no. 5, 1545--1549. \MR{1713513 (2001h:57019)}

\bibitem{KO:2008}
M. Korkmaz and B. Ozbagci, {\it On sections of elliptic fibrations}, Michigan Math. J. \textbf{56} (2008), no. 1, 77–-87. \MR{2433657 (2009f:57043)}

\bibitem{Kod:1967}
K. Kodaira, {\it A certain type of irregular algebraic surfaces}, J. Anal. Math. \textbf{19} (1967), 207--215 \MR{0216521 (35 \#7354)}

\bibitem{KorkSt:2003}
M. Korkmaz and A. Stipsicz, {\it The second homology groups of mapping class groups of orientable surfaces}, Math. Proc. Cambridge Philos. Soc. \textbf{134} (2003), no. 3, 479--489 \MR{1981213 (2004c:57033)}

\bibitem{Kot:1998}
D. Kotschick, {\it Signatures, Monopoles, and Mapping class groups},
Math. Res. Lett. \textbf{5} (1998), no. 1-2, 227--234 \MR{1617905 (99d:57023)}

\bibitem{Mat:1996}
Y. Matsumoto, {\it Lefschetz Fibrations of genus two - a topological approach}, Topology and Teichm$\ddot{u}$ller spaces, 123--148, World Sci.Publ., River Edge, NJ, (1996) \MR{1659687 (2000h:14038)}

\bibitem{Mey:1972}
W. Meyer, {\it Die Signatur von lokalen Koeffizientensystemen und Faserb$\ddot{u}$ndeln}, Bonn. Math. Schr. \textbf{53} (1972) \MR{0305402 (46 \#4532)} 

\bibitem{Mey:1973}
W. Meyer, {\it Die Signatur von Fl$\ddot{a}$chenb$\ddot{u}$ndeln}, Math. Ann. \textbf{201} (1973) 239--264 \MR{0331382 (48 \#9715)}

\bibitem{Morita:1987}
S. Morita, {\it Characteristic classes of surface bundles}, Invent. Math. \textbf{90} (1987) 551--577 \MR{0914849 (89e:57022)}

\bibitem{Oz:2002}
B. Ozbagci, {\it Signatures of Lefschetz fibrations}, Pacific J. Math. \textbf{202} (2002) no. 1, 99--118 \MR{1883972 (2002k:57066)}

\bibitem{Pow:1978}
J. Powell, {\it Two theorems on the mapping class group of a surface}, Proc. Amer. Math. Soc. \textbf{68} (1978), no. 3, 347--350
\MR{0494115 (58 \#13045)}

\bibitem{Sa:2012}
T. Sakasai, {\it Lagrangian mapping class groups from a group homological point of view}, Algebr. Geom. Topol. \textbf{12} (2012), 267--291

\bibitem{St:2002}
A. Stipsicz, {\it Surface bundles with nonvanishing signature}, Acta Math. Hungar. \textbf{95} (2002), no. 4, 299--307 \MR{1909600 (2003d:55014)}

\bibitem{Waj:1999}
B. Wajnryb, {\it An elementary approach to the mapping class group of a surface}, Geom. Topol. \textbf{3} (1999), 405--466 \MR{1726532 (2001a:20059)}

\bibitem{Wall:1969}
C.T.C. Wall, {\it Non-additivity of the signature}, Invent. Math. \textbf{7} (1969) 269--274 \MR{0246311 (39 \#7615)}



\end{thebibliography}
\end{document}